\documentclass[11pt,british,reqno]{amsart}

\usepackage[T1]{fontenc}

\usepackage{babel,eucal,url,amssymb,enumerate,booktabs,amscd,graphics}

\theoremstyle{plain}
\newtheorem{lemma}{Lemma}[section]

\newtheorem{proposition}[lemma]{Proposition}
\newtheorem{corollary}[lemma]{Corollary}
\newtheorem{theorem}[lemma]{Theorem}
\newtheorem{remark}[lemma]{Remark}

\usepackage{color}

\newcommand{\Gtwo}{\ifmmode{{\rm G}_2}\else{${\rm G}_2$}\fi}

 \newcommand{\cyclic}{\mathop{\kern0.9ex{{+}\kern-2.2ex\raise-.28ex\hbox{\Large\hbox
 {$\circlearrowright$}}}}}

\def\sideremark#1{\ifvmode\leavevmode\fi\vadjust{\vbox to0pt{\vss
 \hbox to 0pt{\hskip\hsize\hskip1em
 \vbox{\hsize2.5cm\tiny\raggedright\pretolerance10000
 \noindent #1\hfill}\hss}\vbox to8pt{\vfil}\vss}}}%

\input xy
\xyoption{all}

\newfont{\eusm}{eusm10 scaled \magstep1}
\newfont{\eusmiii}{eusm10 scaled \magstep3}

\title[Invariant contact metric structures]{Invariant contact metric structures on tangent sphere bundles of compact symmetric spaces}

\author[J.~C.~Gonz{\'a}lez-D{\'a}vila]{J.~C.~Gonz{\'a}lez-D{\'a}vila}
\address{Departamento de Matem\'aticas, Estad\'istica e Investigaci\'on
Ope\-ra\-tiva, University of La Laguna, 38200 La Laguna, Tenerife, Spain.}
\email{jcgonza@ull.es}

\thanks{Partially supported by Grant PID2019-105019GB-C21 funded by M-CIN/AEI/10.13039/501100011033 and by ERDF ‘A way of making Europe’.}

\begin{document}

\maketitle

\begin{abstract} A new characterization is provided for the class of compact rank-one symmetric spaces. Such spaces are the only symmetric spaces of compact type for which the standard vector field $\xi^{S}$ on their sphere bundles is Killing with respect to some invariant Riemannian metric. The set of all these metrics is determined, as well as the set of all those invariant contact metric structures with characteristic vector field $\xi^{S}.$ Moreover, on tangent sphere bundles of compact symmetric spaces with rank greater than or equal to two, a family of invariant contact metric structures, which contains the standard structure, is obtained.
\end{abstract}

\section{Introduction}

There is an extensive bibliography on the geometry of the tangent bundle $TM$ over a Riemannian manifold $(M,g)$ and of the tangent sphere bundle $T_{r}M := \{u\in TM\colon g(u,u)= r^{2}\}$ of radius $r>0,$ specially for $r =1,$ equipped with their corresponding standard structures. See D. Blair \cite[Ch. 9]{Bl} and E. Boeckx and L. Vanhecke \cite{BV, BV1}, among others. On the other hand, the induced metric on the unit tangent sphere bundle $T_{1}M$ from a $g$-natural metric, which generalizes the Sasaki and the Cheeger-Gromoll metrics, has been treated by M.T.K. Abbassi, G. Calvaruso and M. Sarih (see, for example, \cite{AC, AS}).

The {\em standard almost Hermitian structure} $(J^{S},g^{S}),$ where $g^{S}$ is the Sasaki metric on $TM,$ is almost K\"ahler \cite{TO}. Then $T_{r}M,$ as the hypersurface $\iota_{r}\colon T_{r}M\to (TM,J^{S},g^{S}),$ inherits the induced contact metric structure $(\frac{1}{2r}\eta^{S},\frac{1}{4r^{2}}\tilde{g}^{S}),$ where $\eta^{S}$ is the one-form metrically equivalent to $\xi^{S} := -J^{S}N,$ $N$ being the outward normal unit vector field to the singular foliation $\{T_{r}M\}_{r\geq 0}.$ The vector field $\xi^{S},$ defined on the punctured bundle $TM\setminus \{\mbox{\rm zero section}\},$ is called the {\em standard vector field} of $TM$ and it is well known that, on the unit tangent bundle $(T_{1}M, \tilde{g}^{S}),$ $\xi^{S}$ is Killing if and only if $(M,g)$ has constant sectional curvature \cite{Tash}.

In the recent paper \cite{JC}, the author has proved the existence of invariant Riemannian metrics for which $\xi^{S}$ is Killing on the tangent sphere bundle, of any radius, of each compact rank-one symmetric space. A first objective that we address in this paper is to prove that there are no metrics satisfying this condition for the case of compact symmetric spaces of rank greater than or equal to two, which leads us to establish in Theorem \ref{tmain} a new characterization for compact rank-one symmetric spaces: 

{\em A symmetric space $G/K$ of compact type has rank one if and only if there exists a $G$-invariant Riemannian metric on the tangent sphere bundle $T_{r}(G/K),$ for some $r>0,$ for which the standard vector field $\xi^{S}$ is Killing.}

Moreover, the $G$-homogeneity of $T_{r}(G/K),$ of any radius $r>0,$ of a compact rank-one symmetric space $G/K,$ allows us to determine in Theorem \ref{mainrank1} {\em all} the $G$-invariant Riemannian metrics $\tilde{\bf g}$ on $T_{r}(G/K)$ for which the pair $(\xi =\frac{1}{\kappa} \xi^{S},\tilde{\bf g}),$ for each constant $\kappa >0,$ is a contact metric structure and hence, according to \cite{JC}, those that are $K$-contact and, consequently, Sasakian.

Our second objective is to give new examples of invariant contact metric structures on the tangent sphere bundle of a compact symmetric space of rank grater than or equal to two. For this we establish some concepts and results about restricted roots and their applications to tangent bundles. Such a study is included in Section $4,$  while in Section $3$ some basic concepts about homogeneous manifolds and their tangent bundles are exposed.

Given a Riemannian symmetric space $G/K$ of compact type, we set a Cartan decomposition ${\mathfrak g} = {\mathfrak k}\oplus {\mathfrak m}$ and a Cartan subspace ${\mathfrak a}$ of ${\mathfrak m}.$ Then, on the open dense subset $D(G/K):= G/H\times W$ of $T(G/K),$ where $H$ is the subgroup $H = \{k\in K\colon {\rm Ad}_{k}u = u\;\mbox{\rm for all}\;u\in {\mathfrak a}\}$ of $K$ and $W$ is a Weyl chamber, we construct a family of $G$-invariant almost Hermitian structures $(J^{q},{\bf g}^{q,a_{0},a_{\lambda}}),$ depending of smooth functions $q\colon {\mathbb R}^{+}\to {\mathbb R}^{+}$ and $a_{0},a_{\lambda}\colon W\to {\mathbb R}^{+},$ for each $\lambda\in \Sigma^{+},$ where $\Sigma^{+}$ is the set of positive restricted roots with respect to ${\mathfrak a}.$ For $q = {\rm Id}_{\mathbb R}$ and $a_{0}=a_{\lambda} = 1,$ for all $\lambda\in \Sigma^{+},$ $(J^{q},{\bf g}^{q, a_{0},a_{\lambda}})$ is the standard almost K\"ahler structure on $G/H\times W$ and, for $q(t) = \tanh t,$ $J^{q}$ coincides with the {\em canonical complex structure} \cite{Szoke1}.

We establish  in Theorem \ref{TJc} necessary and sufficient conditions on the functions $q,$ $a_{0}$ and $a_{\lambda},$ $\lambda\in \Sigma^{+},$ so that $(J^{q},{\bf g}^{q,a_{0},a_{\lambda}})$ can be extended to an almost Hermitian structure on the entire tangent bundle $T(G/K)$ and, in Theorem \ref{tKahler}, we determine when it is moreover almost K\"ahler. The obtained structures depend on smooth functions $q\colon {\mathbb R}^{+}\to {\mathbb R}^{+}$ verifying that $\lim_{t\to 0^{+}}q(t)/t\in {\mathbb R}^{+}$ and there is a unique one that is K\"ahler, precisely when $J^{q}$ is the canonical complex structure, that is, $q(t) = \tanh t.$ The induced structures on the tangent sphere bundle $T_{r}(G/K),$ for each $r>0,$ from these almost K\"ahler structures, allow us to obtain in Theorem \ref{tcontact} a class of $G$-invariant contact metric structures, which contains the standard structure.

 \section{The tangent bundle of a homogeneous Riemannian manifold}

\subsection{Homogeneous Riemannian manifolds} A connected homogeneous manifold $M$ can be described as a quotient manifold $G/K,$ where $G$ is a Lie group, which is supposed to be connected, acting transitively on $M,$ and $K$ is the isotropy subgroup of $G$ at some point $o\in M,$ the {\em origin} of $G/K.$ Denote by $\pi_{K}$ the projection $\pi_{K}\colon G\to G/K,$ $\pi_{K}(a) = aK,$ and by $\tau_{b},$ for each $b\in G,$ the translation $\tau_{b}\colon G/K\to G/K,$ $\tau_{b}(aK) = baK.$ If moreover $g$ is a $G$-invariant Riemannian metric on $M = G/K,$ then $(M,g)$ is said to be a {\em homogeneous Riemannian manifold}. 

When $G$ is compact, there exists an ${\rm Ad}(G)$-invariant inner product $\langle\cdot,\cdot\rangle$ on the Lie algebra ${\mathfrak g}$ of $G$ and we have the reductive decomposition ${\mathfrak g} = {\mathfrak k}\oplus {\mathfrak m},$ ${\mathfrak m}$ being the $\langle\cdot,\cdot\rangle$-orthogonal complement of the Lie algebra ${\mathfrak k}$ of $K.$ Then the restriction $\langle\cdot,\cdot\rangle_{\mathfrak m}$ of $\langle\cdot,\cdot\rangle$ to ${\mathfrak m}$ determines a $G$-invariant metric $g$ on $M$ and $(M = G/K,g)$ is said to be a {\em normal} homogeneous Riemannian manifold.

Fixed a reductive decomposition ${\mathfrak g} = {\mathfrak k}\oplus {\mathfrak m},$ the differential map $(\pi_{K})_{*e}$ of $\pi_{K}$ at the identity element $e$ of $G$ gives an isomorphism of ${\mathfrak m}$ onto $T_{o}(G/K).$ In what follows, $T_{o}(G/K)$ is identified with ${\mathfrak m}$ via $(\pi_{K})_{*e}.$ It is clear that there exists a sufficiently small neighborhood $O_{\mathfrak m}\subset {\mathfrak m}$ of zero in ${\mathfrak m}$ such that $\exp(O_{\mathfrak m})$ is a submanifold of $G$ and the mapping $\pi_{K\mid\exp(O_{\mathfrak m})}$ is a diffeomorphism onto a neighborhood ${\mathcal U}_{o}$ of $o.$ Hence, for each $\mu\in {\mathfrak m},$ a vector field $\mu^{\tau}$ on ${\mathcal U}_{o}$ is defined as
\begin{equation}\label{tautau}
\mu^{\tau}_{(\exp x) K} = (\tau_{\exp x})_{*o}\mu,\quad \mbox{\rm for all $x\in O_{\mathfrak m},$}
\end{equation}
 which satisfies $[\mu_{1}^{\tau},\mu^{\tau}_{2}]_{o} = [\mu_{1},\mu_{2}]_{\mathfrak m}$ (see \cite{N}), where $[\cdot,\cdot]_{\mathfrak m}$ denotes the ${\mathfrak m}$-component of $[\cdot,\cdot].$ 
 
 Let $\alpha\colon {\mathfrak m}\times {\mathfrak m}\to {\mathfrak m}$ be the ${\rm Ad}(K)$-invariant bilinear function that determines the Levi-Civita connection $\nabla$ of $g,$ given by $\alpha(\mu_{1},\mu_{2}) = \nabla_{\mu_{1}}\mu_{2}^{\tau},$ for all $\mu_{1},\mu_{2}\in {\mathfrak m}.$ Then, using the Koszul formula, we have 
\begin{equation}\label{nabla}
\alpha(\mu_{1},\mu_{2}) = \frac{1}{2}[\mu_{1},\mu_{2}]_{\mathfrak m} + {\mathfrak U}(\mu_{1},\mu_{2}),
\end{equation} 
 where ${\mathfrak U}$ is the symmetric bilinear function on ${\mathfrak m}\times{\mathfrak m}$ such that
\[
2\langle{\mathfrak U}(\mu_{1},\mu_{2}),\mu_{3}\rangle = \langle[\mu_{3},\mu_{1}]_{\mathfrak m},\mu_{2}\rangle + \langle[\mu_{3},\mu_{2}]_{\mathfrak m},\mu_{1}\rangle.
\]
 When ${\mathfrak U} = 0,$ $(G/K,g)$ is said to be {\em naturally reductive}.
 
 \subsection{Tangent bundles and tangent spheres bundles} On the trivial vector bundle $G\times {\mathfrak m}$ consider two Lie group actions which commute on it: the left
$G$-action, $l_b \colon (a,x)\mapsto (ba,x)$ and the right
$K$-action $r_k \colon (a,x)\mapsto (ak,{\rm Ad}_{k^{-1}}x)$. Let
$\pi \colon G\times {\mathfrak m}\to G\times_K {\mathfrak m},$ $(a,x)\mapsto [(a,x)],$ be the natural projection for this right
$K$-action. Then $\pi$ is $G$-equivariant and, using that the linear isotropy group $\{(\tau_{k})_{*o} \, : k\in K\}$ acting on $T_{o}(G/K)$ corresponds under $(\pi_{K})_{*e}$ with ${\rm Ad}(K)$ on ${\mathfrak m},$ the mapping $\phi$ given by
\begin{equation}\label{eq.phi}
\phi \colon G\times_K {\mathfrak m}\to T(G/K),
\quad
[(a,x)]\mapsto (\tau_{a})_{*o}x,
\end{equation}
is a $G$-equivariant diffeomorphism which allows us to identify $T(G/K)$ with $G\times_{K} {\mathfrak m}.$

 Because $\pi$ is a submersion, each vector of $T_{[(a,x))]}(G\times_{K}{\mathfrak m}),$ $(a,x)\in G\times {\mathfrak m},$ can be written as $\pi_{*(a,x)}(\mu^{\tt l}_{a},u_{x}),$ for some $\mu\in {\mathfrak g}$ and $u\in {\mathfrak m},$ where $\mu^{\tt l}$ denotes the left-invariant vector field on $G$ such that $\mu^{\tt l}_{e} = \mu\in T_{e}G\cong {\mathfrak g}.$ Put  $\mu = \mu_{\mathfrak k} + \mu_{\mathfrak m}\in {\mathfrak k}\oplus{\mathfrak m}.$ Then, $\pi_{*(a,x)}(\mu^{\tt l}_{a},u_{x})  = \pi_{*(a,x)}( (\mu_{\mathfrak k})^{\tt l}_{a},0)) + \pi_{*(a,x)}(((\mu_{\mathfrak m})^{\tt l}_{a},u_{x}),$ being $u_{x} = \alpha'(0)$ where $\alpha(t) = x + tu.$ Since 
\[
\pi_{*(a,x)}((\mu_{\mathfrak k})^{\tt l}_{a},0) = \frac{d}{dt}_{\mid t = 0}\pi(a\exp t\mu_{\mathfrak k},x) = \frac{d}{dt}_{\mid t = 0}\pi(a,{\rm Ad}_{\exp t\mu_{\mathfrak k}}x) =\pi_{*(a,x)}(0,[\mu_{\mathfrak k},x]_{x}),
\]
it follows that 
\begin{equation}\label{kk}
\pi_{*(a,x)}(\mu^{\tt l}_{a},u_{x})  = \pi_{*(a,x)}((\mu_{\mathfrak m})^{\tt l}_{a},u_{x} + [\mu_{\mathfrak k},x]_{x})
\end{equation}
and we have
\begin{equation}\label{lGm}
 T_{[(a,x)]}(G\times_{K}{\mathfrak m}) = \{\pi_{*(a,x)}(\mu^{\tt l}_{a},u_{x})\colon (\mu,u)\in {\mathfrak m}\times {\mathfrak m}\},\quad (a,x)\in G\times{\mathfrak m}.
 \end{equation}
Next, we obtain the pull-back form $\theta: = \phi^{*}\theta$ on $G\times_{K}{\mathfrak m}$ of the {\em canonical} $1$-form $\theta$ on the tangent bundle $\pi^{T}\colon T(G/K)\to G/K,$ defined by
 \[
 \theta_{u}(X) = g(u,(\pi^{T})_{*u}X),\quad u\in T_{p}(G/K),\quad p\in G/K, \quad  X\in T_{u}(T(G/K)).
 \]
 \begin{lemma} The canonical $1$-form $\theta$ on $G\times_{K}{\mathfrak m}$ is $G$-invariant and it is determined by
 \begin{equation}\label{1canonical}
 \theta_{[(e,x)]}(\pi_{*(e,x)}(\mu,u_{x})) = \langle x,\mu\rangle.
 \end{equation}
 \end{lemma}
\begin{proof} Putting $\Pi = \phi\circ \pi,$ we get
\[
\theta_{[(a,x)]}(\pi_{*(a,x)}(\mu^{\tt l}_{a},u_{x})) = \theta_{(\tau_{a})_{*o}x}(\Pi_{*(a,x)}(\mu^{\tt l}_{a},u_{x})) = g_{\tau_{a}(o)}((\tau_{a})_{*o}x, (\pi^{T}\circ \Pi)_{*(a,x)}(\mu^{\tt l}_{a},u_{x})).
\]
But, under the identification $T_{o}(G/K)\cong {\mathfrak m},$
\[
(\pi^{T}\circ \Pi)_{*(a,x)}(\mu^{\tt l}_{a},u_{x}) = \frac{d}{dt}_{\mid t = 0}(\pi^{T}\circ \Pi)(a\exp t\mu, x + t u) = \frac{d}{dt}_{\mid t = 0}(\pi_{K}(a\exp t\mu))  = (\tau_{a})_{*o}\mu.
\]
Hence, because $g$ is $G$-invariant, we obtain (\ref{1canonical}).
\end{proof}

Finally, we consider the sphere ${\mathcal S}_{\mathfrak m}(r) = \{x\in {\mathfrak m}\colon \langle x,x\rangle = r^{2}\}$ of radius $r>0$ in ${\mathfrak m}$ with respect to the ${\rm Ad}(K)$-invariant inner product $\langle\cdot,\cdot\rangle$ that determines the $G$-invariant metric $g.$ Then, taking into account that for each $[(a,x))]\in G\times_{K}{\mathfrak m}$ we get
\[
\langle x,x\rangle = g_{\tau_{a}(o)}((\tau_{a})_{*o}x, (\tau_{a})_{*o}x) = g_{\tau_{a}(o)}(\phi[(a,x)],\phi[(a,x)]),
 \]
 it follows from (\ref{eq.phi}) that
 \begin{equation}\label{TrSm}
T_{r}(G/K) = \phi(G\times_{K}{\mathcal S}_{\mathfrak m}(r)).
 \end{equation}
 Moreover, from (\ref{lGm}), for each $(a,x)\in G\times {\mathcal S}_{\mathfrak m}(r),$ we have
 \[
 T_{[(a,x)]}(G\times_{K}{\mathcal S}_{\mathfrak m}(r) )= \{\pi_{*(a,x)}(\mu^{\tt l}_{a},u_{x})\colon (\mu,u)\in {\mathfrak m}\times {\mathfrak m}\;\mbox{\rm and}\;\langle u,x\rangle = 0\}.
 \]
 
\subsection{The standard structures}
 
 On the tangent bundle $\pi^{T}\colon TM\to M$ of a Riemannian manifold $(M,g),$ consider the {\em horizontal distribution} ${\mathcal H}\colon u\in TM\to {\mathcal H}_{u}\subset T_{u}TM,$ where ${\mathcal H}_{u}$ is the space of all horizontal lifts of tangent vectors in $T_{p}M,$ $p = \pi^{T}(u),$ obtained by parallel translation with respect to its Levi-Civita connection $\nabla.$ (See for example \cite[Ch. 9]{Bl} and references inside for more information.) Then $TM$ splits into the direct decomposition $T_{u}TM = {\mathcal H}_{u}\oplus {\mathcal V}_{u},$ where ${\mathcal V}_{u}$ is the vertical space ${\mathcal V}_{u} = {\rm Ker}(\pi^{T})_{*u}.$ The standard almost complex structure $J^{S}$ on $TM$ is defined by
\begin{equation}\label{J}
J^{S}X^{\tt h} = X^{\tt v},\quad J^{S} X^{\tt v} = -X^{\tt h},
\end{equation}
for all vector field $X$ on $M$, where $X^{\tt h}$ and $X^{\tt v}$ are the horizontal and vertical lifts of $X,$ respectively. For each $\mu\in T_{u}TM$ denote by $\mu^{\tt ver}$ (resp., $\mu^{\tt hor})$ the ${\mathcal V}_{u}$-component (resp., ${\mathcal H}_{u}$-component) of $\mu.$ Then the  connection map $K$ of $\nabla$ is defined by $K_{u}(\mu) = \iota_{u}(\mu^{\tt ver}),$ where $\iota_{u}$ is the projection $\iota_{u}\colon T_{u}TM\to T_{p}M$ such that $\iota_{u}(\mu) = 0,$ for all $\mu\in {\mathcal H}_{u}$ and $\iota_{u}(v^{\tt v}_{u}) = v.$ This satisfies $K_{X_{p}}(X_{*p}u) = \nabla_{u}X$ and the Sasaki metric $g^{S}$ on $TM$ is given by
\begin{equation}\label{Sasaki}
g^{S}(\mu_{1},\mu_{2}) = (g((\pi^{T})_{*}\mu_{1},(\pi^{T})_{*}\mu_{2}) + g(K(\mu_{1}),K(\mu_{2})))\circ \pi^{T},
\end{equation}
for all $\mu_{1},\mu_{2}\in T(TM)$ such that $\pi^{T}\mu_{1} = \pi^{T}\mu_{2}.$ Then $g^{S}(X^{\tt h},Y^{\tt h}) = g^{S}(X^{\tt v},Y^{\tt v}) = g(X,Y)\circ \pi^{T}$ and $g^{S}(X^{\tt v},Y^{\tt h}) = 0,$ for all vector fields $X,Y$ on $M.$ Therefore, $g^{S}$ is a Hermitian metric with respect to $J^{S}.$ The pair $(J^{S},g^{S})$ is known as the {\em standard almost Hermitian structure}.
 
 Now suppose that $(M= G/K,g)$ is a homogeneous Riemannian manifold and denote by $\widetilde{\pi}$ the projection $\widetilde{\pi}\colon G\times_{K}{\mathfrak m}\to G/K.$ Then, $\widetilde{\pi} = \pi^{T}\circ \phi.$ Let $\widetilde{\mathcal V}$ and $\widetilde{\mathcal H}$ be the distributions $\widetilde{\mathcal V}:=(\phi^{-1})_{*}({\mathcal V}) = {\rm Ker}\;\widetilde{\pi}_{*}$ and $\widetilde{\mathcal H} :=(\phi^{-1})_{*}({\mathcal H})$ in $T(G\times_{K}{\mathfrak m}).$ Because $\nabla$ is $G$-invariant, ${\mathcal V},$ $\widetilde{\mathcal V},$ ${\mathcal H}$ and $\widetilde{\mathcal H}$ are $G$-invariant distributions. 
 
 \begin{lemma}\label{lGK} Under the $G$-equivariant diffeomorphism $\phi\colon G\times_{K}{\mathfrak m}\to T(G/K),$ the standard almost Hermitian structure $(J^{S},g^{S})$ on $G\times_{K}{\mathfrak m}$ is determined by
 \begin{equation}\label{JG}
\begin{array}{l}
J^{S}_{[(e,x)]}\pi_{*(e,x)}(\mu,u_{x}) = \pi_{*(e,x)]}(-u-\alpha(\mu,x),\mu_{x} + \alpha(u + \alpha(\mu,x),x)_{x}),\\[0.4pc]
g^{S}(\pi_{*(e,x)}(\mu,u_{x}),\pi_{*(e,x)}(\nu,v_{x})) = \langle\mu,\nu\rangle + \langle u + \alpha(\mu,x), v + \alpha(\nu,x)\rangle,
\end{array}
\end{equation}
for all $x,\mu, \nu, u, v\in {\mathfrak m}.$ Moreover, the outward normal unit vector field $N$ and the standard vector field $\xi^{S}$ on $G\times_{K}{\mathcal S}_{\mathfrak m}(r),$ for each $r>0,$ are determined by
\begin{equation}\label{Nxi}
N_{[(e,x)]} = \frac{1}{r}\pi_{*(e,x)}(0,x_{x}),\quad \xi^{S}_{[(e,x)]} = \frac{1}{r}\pi_{*(e,x)}(x,-\alpha(x,x)_{x}),\quad \mbox{for all $x\in {\mathcal S}_{\mathfrak m}(r).$}
\end{equation}
\end{lemma}
 \begin{proof} Taking into account that $(\widetilde{\pi}\circ \pi)(a,x) = \pi_{K}(a),$ for all $(a,x)\in G\times {\mathfrak m},$ and using (\ref{lGm}), we have $\widetilde{\mathcal V}_{[(a,x)]} = \{\pi_{*(a,x)}(0,u_{x})\;:\; u\in {\mathfrak m}\}$ and, moreover,
\begin{equation}\label{verticallift}
\mu^{\tt v}_{[(e,x)]} = \pi_{*(e,x)}(0,\mu_{x}),\quad \mu\in {\mathfrak m}.
\end{equation}

For each $x\in {\mathfrak m},$ consider $x^{\tau}$ as a local section of $T(G/K).$ Then
\begin{equation}\label{xi+}
(x^{\tau})_{*o}\mu = \frac{d}{dt}_{\mid t = 0}x^{\tau}(\exp t\mu) = \frac{d}{dt}_{\mid t = 0}(\tau_{\exp t\mu_{*o}})x = \phi_{*[(e,x)]}(\pi_{*(e,x)}(\mu,0)),\quad \mu\in {\mathfrak m}.
\end{equation}
Since $K_{x}((x^{\tau})_{*o}\mu) = \nabla_{\mu}x^{\tau} = \alpha(\mu,x),$ we have $((x^{\tau})_{*o}\mu)^{\tt ver} = \alpha(\mu,x)_{x}\in T_{x}T_{o}M.$ Then from (\ref{xi+}), the vertical component $(\pi_{*(e,x)}(\mu,0))^{\tt ver}$ of $\pi_{*(e,x)}(\mu,0)$ in the decomposition $T_{[(e,x)]}(G\times_{K}{\mathfrak m}) = \widetilde{\mathcal H}_{[(e,x)]}\oplus\widetilde{\mathcal V}_{[(e,x)]}$ is  given by
\[
(\pi_{*(e,x)}(\mu,0))^{\tt ver} = (\phi^{-1})_{*x}\alpha(\mu,x)_{x} = \pi_{*(e,x)}(0,\alpha(\mu,x)_{x}).
\]
Hence, applying (\ref{verticallift}), it follows, for all $(\mu,u)\in {\mathfrak m}\times {\mathfrak m},$ that
\[
(\pi_{*(e,x)}(\mu,u_{x}))^{\tt ver} = \pi_{*(e,x)}(0,u_{x} + \alpha(\mu,x)_{x}),\quad (\pi_{*(e,x)}(\mu,u_{x}))^{\tt hor}  =  \pi_{*(e,x)}(\mu,-\alpha(\mu,x)_{x}).
\]
 On the other hand, because $\widetilde{\pi}_{*[(e,x)]}(\pi_{*(e,x)}(\mu,u_{x})) = \mu\in {\mathfrak m}\cong T_{o}M,$ 
\begin{equation}\label{vh}
\mu^{\tt h}_{[(e,x)]} = (\pi_{*(e,x)}(\mu,u_{x}))^{\tt hor}_{[(e,x)]} = \pi_{*(e,x)}(\mu, -\alpha(\mu,x)_{x}).
\end{equation}
Then $\tilde{\mathcal H}$ can be expressed as $\widetilde{\mathcal H}_{[(e,x)]} = \{\pi_{*(e,x)}(\mu,-\alpha(\mu,x)_{x})\;:\;\mu\in {\mathfrak m}\}.$ Now, applying (\ref{verticallift}) and (\ref{vh}) in (\ref{J}) and (\ref{Sasaki}), it follows (\ref{JG}) and hence, also (\ref{Nxi}).
\end{proof}

 \section{The tangent bundle of a symmetric space of compact type}\label{tres}

  Let $G/K$ be a Riemannian symmetric space of compact type and let ${\mathfrak g} = {\mathfrak k}\oplus {\mathfrak m}$ be the decomposition of ${\mathfrak g}$ induced by the symmetric structure. Then, $[{\mathfrak m},{\mathfrak m}]\subset{\mathfrak k}.$ Because $G$ is compact and semisimple, the Killing form $B\colon{\mathfrak g}\times{\mathfrak g}\to {\mathfrak g},$ $(\mu,\nu)\mapsto {\rm tr}({\rm ad}_{\mu}\circ {\rm ad}_{\nu})$ is negative definite. In the sequel, we will assume that $G/K$ is equipped with the $G$-invariant metric that determines the inner product $\langle\cdot,\cdot\rangle:=-cB$ on ${\mathfrak m},$ for an arbitrary $c>0.$
  
  Let us review some details about restricted roots. (For the corresponding theory, see for example \cite[Ch. III and Ch. VII]{He} and \cite[Section V.2]{Loos}.) If $G/K$ has rank ${\mathbf r},$ then, fixing an ${\mathbf r}$-dimensional Cartan subspace ${\mathfrak a}$ of ${\mathfrak m},$ there exists a subalgebra ${\mathfrak t}_{\mathfrak a}$ of ${\mathfrak g}$ containing ${\mathfrak a}$ such that its complexification ${\mathfrak t}_{\mathfrak a}^{\mathbb C}$ is a Cartan subalgebra of ${\mathfrak g}^{\mathbb C}$ and we have the root space decomposition
\[
{\mathfrak g}^{\mathbb C} = {\mathfrak t}_{\mathfrak a}^{\mathbb C} \oplus \sum_{\alpha\in \Delta}{\mathfrak g}_{\alpha},
\]
where $\Delta$ is the root system of ${\mathfrak g}^{\mathbb C}$ with respect to ${\mathfrak t}_{\mathfrak a}^{\mathbb C}$  and, for each $\alpha\in \Delta,$
\[ 
{\mathfrak g}_{\alpha} = \{\xi\in {\mathfrak g}^{\mathbb C}\;:\;{\rm ad}_{t}\xi = \alpha(t)\xi,\;t\in {\mathfrak t}^{\mathbb C}\}\quad{\rm with}\quad \dim_{\mathbb C}{\mathfrak g}_{\alpha} = 1.
\] 
 
 To the set $\Sigma = \{\lambda\in ({\mathfrak a}^{\mathbb C})^{*}\colon\lambda = \alpha_{\mid {\mathfrak a}^{\mathbb C}},\,\alpha\in \Delta\setminus \Delta_{0}\}$ of non-vanishing restrictions of roots to ${\mathfrak a}^{\mathbb C},$ where $\Delta_{0} = \{\alpha\in \Delta\colon\alpha_{\mid {\mathfrak a}^{\mathbb C}} = 0\},$ is known as the set of {\em restricted roots} with respect to ${\mathfrak a}.$ Denote by $\Sigma^{+}$ the subset of positive restricted roots and by $m_{\lambda}$ the multiplicity of $\lambda\in \Sigma^{+},$ that is, $m_{\lambda} = {\rm card}\{\alpha\in \Delta\colon \alpha_{\mid {\mathfrak a}^{\mathbb C}} = \lambda\}.$ For each $\lambda\in \Sigma^{+},$ put 
  \[
{\mathfrak m}_\lambda =
\big\{\eta\in{\mathfrak m}:\operatorname{ad}^2_w(\eta)
=\lambda^2(w)\eta,\ \forall w\in{\mathfrak a}\big\}, \quad
{\mathfrak k}_\lambda =
\big\{\zeta\in{\mathfrak k}: \operatorname{ad}^2_w(\zeta)=\lambda^2(w)\zeta,\
\forall w\in{\mathfrak a}\big\}.
\]
Then ${\mathfrak m}_{\lambda}={\mathfrak m}_{-\lambda}$,
${\mathfrak k}_{\lambda}={\mathfrak k}_{-\lambda}$,
${\mathfrak m}_0={\mathfrak a}$ and ${\mathfrak k}_0$ equals to the centralizer ${\mathfrak h}$ of ${\mathfrak a}$ in
${\mathfrak k},$ that is,
\[
{\mathfrak h} = \{u\in {\mathfrak k}\;\colon [u,{\mathfrak a}] = 0\}.
\]
 By~\cite[Ch.\ VII, Lemma 11.3]{He}, the following
decompositions are direct and orthogonal:
\[
{\mathfrak m}={\mathfrak a}\oplus\sum_{\lambda\in\Sigma^+}{\mathfrak m}_\lambda,
\qquad
{\mathfrak k}={\mathfrak h}\oplus
\sum_{\lambda\in\Sigma^+}{\mathfrak k}_\lambda.
\]
Moreover, from \cite[Ch. VII, Lemma 2.3]{He}, we obtain that, for any $\xi_{\lambda}\in {\mathfrak m}_{\lambda},$ $\lambda\in\Sigma^+$, there exists a unique vector $\zeta_{\lambda}\in{\mathfrak k}_\lambda$ such that $[u,\xi_{\lambda}] = i\lambda(u)\zeta_{\lambda}$ and $[u,\zeta_{\lambda}] = -i\lambda(u)\xi_{\lambda},$ for all $u\in {\mathfrak a}.$ Then,
\begin{equation}\label{eq.ms4.3}
[u,\xi_\lambda]=  -\lambda_{\mathbb R}(u)\zeta_\lambda,
\quad [u,\zeta_\lambda]= \lambda_{\mathbb R}(u)\xi_\lambda,
\end{equation}
where $\lambda_{\mathbb R}$ is the linear function $\lambda_{\mathbb R} \colon {\mathfrak a}\to{\mathbb R}$,
$\lambda\in\Sigma^+$, defined by the relation ${\rm i}\lambda_{\mathbb R}=\lambda$. Note that, since the Lie algebra ${\mathfrak g}$ is compact, $\lambda({\mathfrak a})\subset {\rm i} {\mathbb R}$. In particular, $\dim{\mathfrak m}_{\lambda} = \dim{\mathfrak k}_{\lambda} = m_{\lambda}.$

Each restricted root $\lambda$ defines a hyperplane $\lambda(w) = 0$ in the vector space ${\mathfrak a}$ and these hyperplanes divide ${\mathfrak a}$ into finitely many connected components called Weyl chambers (see \cite[Ch. VII]{He}). Next, we consider the {\em Weyl chamber} $W$ given by
\[
 W = \{w\in {\mathfrak a}\;:\; \lambda_{\mathbb R}(w)>0,\;\mbox{\rm for each}\;\lambda\in \Sigma^{+}\}.
 \]
 Because each  ${\rm Ad}(K)$-orbit in ${\mathfrak m}$ intersects the Cartan subspace ${\mathfrak a},$ it follows that ${\rm Ad}(K){\mathfrak a}= {\mathfrak m}$ and the open connected subset ${\mathfrak m}^{R}:= {\rm Ad}(K)(W)$ is dense in ${\mathfrak m},$ called the {\em set of regular points} of ${\mathfrak m}.$ Hence $D(G/K):=\phi(G\times_{K}{\mathfrak m}^{R})$ is an open dense subset of $T(G/K).$ 

 Let $H$ be the closed subgroup of $K,$ with Lie algebra ${\mathfrak h},$ defined by
 \[
 H = \{k\in K\colon {\rm Ad}_{k} u = u,\;\mbox{\rm for all}\; u\in {\mathfrak a}\}
 \] 
 and consider the projection $\pi_{H}\times{\rm id}:G\times W\to G/H\times W.$ It follows that the mapping
\[
\chi\colon G/H\times W\to G\times_{K}{\mathfrak m}^{R},\quad (gH,w)\mapsto [(g,w)]
\]
is a well-defined $G$-equivariant diffeomorphism onto the open dense subset $G\times_{K}{\mathfrak m}^{R}$ of $G\times_{K}{\mathfrak m}$ with inverse mapping $\chi^{-1}$ given by $\chi^{-1}([(g,x)]) = (gkH,w),$ for all $g\in G,$ $k\in K$ and $w\in W,$ where $x = {\rm Ad}_{k}w\in {\mathfrak m}^{R}.$ Then $G/H\times W$ is identified with $D(G/K)$ via the $G$-equivariant diffeomorphism $\phi\circ \chi.$ 

Since the Cartan subspace ${\mathfrak a}$ of ${\mathfrak m}$ is ${\rm Ad}(H)$-invariant and we get ${\rm ad}^{2}_{w}\circ {\rm Ad}_{h} = {\rm Ad}_{h}\circ {\rm ad}^{2}_{w},$ for all $w\in {\mathfrak a}$ and $h\in H,$ the subspaces ${\mathfrak m}_{\lambda}$ and ${\mathfrak k}_{\lambda}$ are all ${\rm Ad}(H)$-invariant and ${\mathfrak g} = {\mathfrak h}\oplus \overline{\mathfrak m}$ is a reductive decomposition for the quotient manifold $G/H,$ where $\overline{\mathfrak m} = {\mathfrak m}\oplus \sum_{\lambda\in\Sigma^+}{\mathfrak k}_\lambda.$

Fixing in each subspace
${\mathfrak m}_\lambda$, $\lambda\in\Sigma^+$, some $\langle\cdot,\cdot\rangle$-orthonormal basis
$\{\xi_\lambda^s, s=1, \dotsc ,m_\lambda\},$ we take the unique basis, for each
$\lambda\in\Sigma^+,$ $\{\zeta_\lambda^s,\, s=1, \dotsc,m_\lambda\}$ of
${\mathfrak k}_\lambda$ such that for each pair
$(\xi_\lambda^s,\zeta_\lambda^s),$ the condition~(\ref{eq.ms4.3}) holds.
Then the basis $\{\zeta_\lambda^s,
s=1, \dotsc,m_\lambda\}$, $\lambda\in\Sigma^+$, of ${\mathfrak k}_\lambda$,
is also orthonormal. On the Cartan subspace ${\mathfrak a}$ we consider some orthonormal basis $\{X_1,\dotsc,X_{\mathbf r}\}$ and its associated coordinates $(x_1,\dotsc,x_{\mathbf r})$ on $W,$ that is, $x_{j}(w) = w_{j},$ $j = 1,\dots, {\mathbf r},$ where $w=\sum_{j=1}^{\mathbf r}w_{j}X_{j}\in W.$ Moreover, we shall also take $(x_{1},\dots ,x_{\mathbf r})$ as the smooth functions such that $x_{j}(aH,w) = w_{j}.$ Hence, under the identification of $T_{(o_{H},w)}(G/H\times W)$ with $\overline{\mathfrak m}\times T_{w}W$ via $(\pi_{H}\times {\rm id})_{*(e,w)},$ where $o_{H} = \{H\}$ is the origin of $G/H,$  it follows that the family
\[
\{(X_{j},0),(0,\frac{\partial}{\partial x_{j}}(x)),\;j = 1,\dots ,{\bf r};\;(\xi^{s}_{\lambda},0),(\zeta^{s}_{\lambda},0),\;s = 1,\dots ,m_{\lambda},\;\lambda\in \Sigma^{+}\}
\]
is a basis of $T_{(o_{H},w)}(G/H\times W).$ Moreover, using (\ref{kk}) the differential $\chi_{*(o_{H},w)}$ of $\chi$ at $(o_{H},w)\in G/H\times W,$ satisfies
\begin{equation}\label{f+}
\chi_{*(o_{H},w)}(\xi,u_{w}) = \pi_{*(e,w)}(\xi,u_{w}) = \pi_{*(e,w)}(\xi_{\mathfrak m}, u_{w} + [\xi_{\mathfrak k},w]_{w}),
\end{equation}
for all $\xi\in \overline{\mathfrak m}$ and $u_{w}\in T_{w}W,$ we have from (\ref{eq.ms4.3}) that
\begin{equation}\label{f++}
\begin{array}{l}
\chi_{*(o_{H},w)}(X_{j},0)  =  \pi_{*(e,w)}(X_{j}, 0),\quad \chi_{*(o_{H},w)}(0, \frac{\partial}{\partial x_{j}}(w))  = \pi_{*(e,w)}(0,(X_{j})_{w}),\\[0.4pc]
\chi_{*(o_{H},w)}(\xi^{s}_{\lambda},0)  =  \pi_{*(e,w)}(\xi^{s}_{\lambda}, 0),\quad \chi_{*(o_{H},w)}(\zeta^{s}_{\lambda},0) = \pi_{*(e,w)}(0,-\lambda_{\mathbb R}(w)(\xi^{s}_{\lambda})_{w}).
\end{array}
\end{equation}

Next, let ${\mathcal S}^{R}_{\mathfrak m}(r) = {\mathcal S}_{\mathfrak m}(r)\cap {\mathfrak m}^{R}$ be the set of regular points  of ${\mathcal S}_{\mathfrak m}(r),$ for some $r>0,$ and denote by $D_{r}(G/K)$ the open dense subset $D_{r}(G/K)= T_{r}(G/K)\cap D(G/K)$ of $T_{r}(G/K).$ Put ${\mathcal S}_{W}(r) = {\mathcal S}_{\mathfrak a}(r)\cap W\subset {\mathcal S}^{R}_{\mathfrak m}(r),$ where ${\mathcal S}_{\mathfrak a}(r)$ is the $r$-sphere in ${\mathfrak a}$ with respect to $\langle\cdot,\cdot\rangle.$ 
  
  \begin{proposition}\label{ll1} We have
 \[
 D_{r}(G/K) = (\phi\circ \chi)(G/H\times {\mathcal S}_{W}(r)).
 \] 
  \end{proposition}
\begin{proof} From (\ref{TrSm}), $D_{r}(G/K) = \phi(G\times_{K}{\mathcal S}_{\mathfrak m}^{R}(r)).$ Let us then see that the equality
\begin{equation}\label{SS}
\chi(G/H\times {\mathcal S}_{W}(r)) = G\times_{K}{\mathcal S}^{R}_{\mathfrak m}(r)
\end{equation}   
holds. Clearly, if $(aH,w)\in G/H\times {\mathcal S}_{W}(r)$ then $\chi(aH,w) = [(a,w)]\in G\times_{K}{\mathcal S}^{R}_{\mathfrak m}(r).$ Moreover, if $[(a,x)]\in G\times_{K}{\mathcal S}^{R}_{\mathfrak m}(r)$ then there exists $k\in K$ and $w\in W$ such that $x = {\rm Ad}_{k}w.$ Because the inner product $\langle\cdot,\cdot\rangle$ is ${\rm Ad}(K)$-invariant, $w$ belongs to ${\mathcal S}_{W}(r)$ and then $(akH,w)\in G/H\times {\mathcal S}_{W}(r).$ But, $\chi(akH,w) = [(ak,w)] = [(a,{\rm Ad}_{k}w)] = [(a,x)].$ This proves equality (\ref{SS}).
\end{proof}

  \section{The standard vector field}
  
  Next, studying separately the cases of rank one and rank greater than or equal to two, we determine necessary and sufficient conditions for the existence of $G$-invariant metrics on $G/H\times {\mathcal S}_{W}(r)$ for which $\xi^{S}$ is a Killing vector field. 
  
  Let us first obtain the standard almost Hermitian structure $(J^{S},g^{S})$ on $G/H\times W$ by its identification with $D(G/K)$ via $\phi\circ \chi,$ i.e. $J^{S} = (\phi\circ \chi)^{-1}_{*}J^{S}(\phi\circ \chi)_{*}$ and $g^{S} = (\phi\circ \chi)^{*}g^{S},$ and $\xi^{S} =\chi^{-1}_{*}\xi^{S}$ on $G/H\times {\mathcal S}_{W}(r),$ for each $r>0.$  
  
 \begin{proposition}\label{estandard} The standard almost Hermitian structure $(J^{S},g^{S})$ on $G/H\times W$ is $G$-invariant and it is determined at $(o_{H},w),$ $w\in W,$ by 
$$
\begin{array}{l}
J^{S}_{(o_{H},w)}(X_{j},0)  =  (0, \frac{\partial}{\partial x_{j}}), \qquad J^{S}_{(o_{H},w)}(0,\frac{\partial}{\partial x_{j}}) = (-X_{j}, 0),\\[0.4pc]
J^{S}_{(o_{H},w)}(\xi^{s}_{\lambda},0) =  (-\frac{1}{\lambda_{\mathbb R}(w)}\zeta_{\lambda}^{s},0), \quad  J^{S}_{(o_{H},w)}(\zeta^{s}_{\lambda},0)  =  (\lambda_{\mathbb R}(w)\xi^{s}_{\lambda},0);\\[0.4pc]
g^{S}_{(o_{H},w)}((X_{j},0),(X_{j},0))  =  g^{S}_{(o_{H},w)}((0,\frac{\partial}{\partial x_{j}}),(0,\frac{\partial}{\partial x_{j}}) ) =  1,\\[0.4pc]
g^{S}_{(o_{H},w)}((\xi^{s}_{\lambda},0),(\xi^{s}_{\lambda},0))  =  \frac{1}{\lambda^{2}_{\mathbb R}(w)}g^{S}_{(o_{H},w)}((\zeta^{s}_{\lambda},0),(\zeta^{s}_{\lambda},0))  =  1,
\end{array}
$$
where $j = 1,\dots ,{\bf r}$ and $s = 1,\dots, m_{\lambda},$ $\lambda\in \Sigma^{+},$ the rest of components of $g^{S}$ being zero. Moreover, the standard vector field $\xi^{S}$ on $G/H\times {\mathcal S}_{W}(r),$ for each $r>0,$ is determined by 
\begin{equation}\label{xixi}
\xi^{S}_{(o_{H},w)} = (\frac{1}{r}w,0),\quad w\in {\mathcal S}_{W}(r).
\end{equation}
\end{proposition}
\begin{proof} Because $G/K$ is a Riemannian symmetric space, the bilinear function $\alpha$ in (\ref{nabla}) is the zero function. Then, from Lemma \ref{lGK} and taking into account (\ref{lGm}), the result follows directly from (\ref{f++}).
\end{proof}
  
 \subsection{The Case ${\rm rank}\;G/K = 1$} Then $\dim {\mathfrak a} =1$ and the restricted root system $\Sigma$ of ${\mathfrak m}$, with respect to ${\mathfrak a},$ is either $\Sigma = \{\pm\varepsilon\},$ for the Euclidean sphere ${\mathbb S}^{n}$ and the real projective space ${\mathbb R}{\mathbf P}^{n},$ or $\Sigma = \{\pm\varepsilon,\pm\frac{1}{2}\varepsilon\},$ for the complex and quaternionic projective spaces ${\mathbb C}{\mathbf P}^{n}$ and ${\mathbb H}{\mathbf P}^{n}$ and for the Cayley plane ${\mathbb C}a{\mathbf P}^{2},$ where $\varepsilon$ is some pure imaginary linear function $\varepsilon = {\rm i}\varepsilon_{\mathbb R}$ on ${\mathfrak a}^{\mathbb C}$ (see \cite[Lemma 3.2]{JC}). 
 
 Denote by $X$ the unique (basis) vector $X\in {\mathfrak a}$ such that $\varepsilon_{\mathbb R}(X) = 1.$ Multiplying the inner product $\langle\cdot,\cdot\rangle$ by a positive constant, we can assume that $\langle X,X\rangle = 1.$ Then the Weyl chamber $W$ in ${\mathfrak a}$ is given by 
\[   
 W = \{tX\colon t\in {\mathbb R}^{+}\}
 \]
 which is naturally identified with ${\mathbb R}^{+}.$ The open dense subset of regular points ${\mathfrak m}^{R}\subset {\mathfrak m}$ is in this case ${\mathfrak m}\setminus \{0\}$ and $D(G/K)= \phi(G\times_{K}{\mathfrak m}^{R})$ is the punctured tangent bundle $T(G/K)\setminus\{\mbox{\rm zero section}\}.$ Hence, $D(G/K) = (\phi\circ \chi)(G/H\times {\mathbb R}^{+})$ and  $D_{r}(G/K) = T_{r}(G/K).$ Because ${\mathcal S}_{W}(r)= \{r\},$ $T_{r}(G/K)$ can be identified with the quotient space $G/H$ and, from (\ref{xixi}), $\xi^{S}$ on $G/H$ is the $G$-invariant vector field such that $\xi^{S}_{o_{H}} = X.$

The quotient expressions $G/K$ as symmetric spaces and the isotropy subgroups $H$ of $G$ for $T_{r}(G/K)$ appear in Table I.

    \bigskip
\begin{tabular}{llllll}
\multicolumn{6}{c}{Table I. Compact rank-one symmetric spaces}\\
\hline\noalign{\smallskip}
 & $G/K$
& {$\dim$}
& $m_\varepsilon$
& $m_{\varepsilon/{\scriptscriptstyle 2}}$ & $H$ \\
\noalign{\smallskip}\hline\noalign{\smallskip}
$\mathbb{S}^n,$\,$(n\geq 2)$ & $\mathrm{SO}(n +1)/\mathrm{SO}(n)$ & $n$ & $n-1$ & $0$ & ${\rm SO}(n-1)$\\ \smallskip
${\mathbb R}{\mathbf P}^n,$ \;$(n\geq 2)$ & $\mathrm{SO}(n{+}1)/{\rm S}({\rm O}(1)\times {\rm O}(n))$ & $n$ & $n{-}1 $ & $0$ & ${\rm S}({\rm O}(1)\times{\rm O}(n-1))$ \\ \smallskip
${\mathbb C}{\mathbf P}^n$\, $(n\geq 2)$ & $\mathrm{SU}(n{+}1)/\mathrm{S}(\mathrm{U}(1){\times}
\mathrm{U}(n))$ & $2n$
& $1$ & $2n-2$ & ${\rm S}({\rm U}(1)\times{\rm U}(n-1))$ \\ \smallskip
${\mathbb H}{\mathbf P}^n$ \,$(n\geq 1)$ & ${\mathrm{Sp}(n{+}1)/\mathrm{Sp}(1){\times} \mathrm{Sp}(n)}$ & $4n$
& $3$ & $4n{-}4$ & ${\rm Sp}(n-1)$ \\
\smallskip
${\mathbb C\mathrm a}{\mathbf P}^2$ & ${\mathrm{F_4}/\mathrm{Spin(9)}}$ & $16$ & $7$
& 8 & $\mathrm{Spin}(7)$ \\
\hline
\end{tabular}

\vspace{0.5cm}

 In order to simplify notations, we put ${\mathfrak m}_{\varepsilon/2} = 0$ and ${\mathfrak k}_{\varepsilon/2} = 0$ if $\varepsilon/2$ is not in $\Sigma.$ Then, taking into account that the decomposition $\overline{\mathfrak m} = {\mathfrak a}\oplus{\mathfrak m}_{\varepsilon}\oplus{\mathfrak m}_{\varepsilon/2}\oplus{\mathfrak k}_{\varepsilon}\oplus{\mathfrak k}_{\varepsilon/2}$ is ${\rm Ad}(H)$-irreducible \cite{JC}, any $G$-invariant Riemannian metric $\tilde{\mathbf g}$ on $G/H$ is determined at the origin $o_{H}$ by 
\begin{equation}\label{Bb}
\tilde{\mathbf g}_{o_{H}} = a_{0}^{2}\langle\cdot,\cdot\rangle_{\mathfrak a} + a_{\varepsilon}\langle\cdot,\cdot\rangle_{{\mathfrak m}_{\varepsilon}} +  a_{\varepsilon/2}\langle\cdot,\cdot\rangle_{{\mathfrak m}_{\varepsilon/2}}    +  b_{\varepsilon}\langle\cdot,\cdot\rangle_{{\mathfrak k}_{\varepsilon}} + b_{\varepsilon/2}\langle\cdot,\cdot\rangle_{{\mathfrak k}_{\varepsilon/2}},
\end{equation}
where $a_{0},a_{\lambda}$ and $b_{\lambda},$ for $\lambda\in \Sigma^{+},$ are positive constants. Therefore, it follows, using Proposition \ref{estandard}, that the induced metric $\tilde{g}^{S}$ of the Sasaki metric $g^{S}$ on $T_{r}(G/K) = G/H\times\{r\} = G/H$ is given by
\begin{equation}\label{inducedgs}
\begin{array}{lcl}
\tilde{g}^{S}_{o_{H}}= \left\{\begin{array}{lcl} \langle\cdot,\cdot\rangle_{\mathfrak m} + r^{2}\langle\cdot,\cdot\rangle_{{\mathfrak k}_{\varepsilon}} &\mbox{if} & G/K = {\mathbb S}^{n}\;\mbox{or}\;\; {\mathbb R}{\mathbf P}^{n},\\[0.4pc]
 \langle\cdot,\cdot\rangle_{\mathfrak m} + r^{2}\langle\cdot,\cdot\rangle_{{\mathfrak k}_{\varepsilon}} + \frac{r^{2}}{4}\langle\cdot,\cdot\rangle_{{\mathfrak k}_{\varepsilon/2}} & \mbox{if} & G/K =  {\mathbb C}{\mathbf P}^{n},\; {\mathbb H}{\mathbf P}^{n}\;\mbox{or}\; \;{\mathbb C}a{\mathbf P}^{n}.
 \end{array}
 \right.
 \end{array}
 \end{equation}
  \subsection{The case ${\rm rank}\;G/K = 2$} \label{dos} Denote the two rays in the Cartan subspace ${\mathfrak a},$ $\dim{\mathfrak a} = 2,$ delineating the Weyl chamber $W$ by $R_{1}$ and $R_{2}$ and by $\theta_{\rm max}$ the angle between $R_{1}$ and $R_{2}.$  Then $\theta_{\rm max}$ is equal to $\pi/3,$ $\pi/4,$ $\pi/4$ or $\pi/6,$ according to whether $\Sigma$ is of type $A_{2},$ $B_{2},$ $BC_{2}$ or $G_{2},$ respectively, (see, for example, \cite{Klein1}) and ${\mathcal S}_{W}(r)$ can be expressed as
 \[
 {\mathcal S}_{W}(r) = \{w(\theta) = r(\cos\theta X_{1} + \sin\theta X_{2})\colon 0 < \theta < \theta_{\rm max}\}.
 \]
 Then $\overline{{\mathcal S}_{W}(r)}\setminus {\mathcal S}_{W}(r) = \{\overline{w}_{1} =w(0),\overline{w}_{2} = w(\theta_{\rm max})\}$ and the open dense subset $D_{r}(G/K)$ of $T_{r}(G/K),$ for each $r>0,$ can be identified with the product manifold $G/H\times ]0,\theta_{\rm max}[.$  From (\ref{xixi}), $\xi^{S}$ is determined on $G/H\times ]0,\theta_{\rm max}[$ by $\xi^{S}_{(o_{H},\theta)} = (\cos\theta X_{1} + \sin\theta X_{2},0)$ and
 \[
 \{\xi^{S}_{(o_{H},\theta)},(\cos\theta X_{2}-\sin\theta X_{1},0),(0,\frac{d}{d\theta}), (\xi^{s}_{\lambda},0),(\zeta^{s}_{\lambda},0)\},
 \]
 for all $s = 1,\dots, m_{\lambda},$ $\lambda\in \Sigma^{+},$ is a basis of $\overline{\mathfrak m}\times T_{\theta}]0,\theta_{\rm max}[\cong T_{(o_{H},w(\theta))}(G/H\times {\mathcal S}_{W}(r)).$ 
 
 \subsection{An example: $SU(3)/SO(3)$}\label{examples} With the purpose of illustrating what was commented in Section \ref{tres} and in the previous Subsection \ref{dos}, we analyze the 5-dimensional compact rank-two symmetric space $G/K = SU(3)/SO(3),$ whose restricted root system $\Sigma$ if of the type $A_{2}.$ (See \cite{CNV} and \cite{Klein2} for more details about this space).
  
  Let ${\mathfrak g} = {\mathfrak s}{\mathfrak u}(3)$ and ${\mathfrak k} = {\mathfrak s}{\mathfrak o}(3)$ be the Lie algebras of $G$ and $K,$ i.e. the spaces of traceless skew-Hermitian complex and skew-symmetric real $3\times 3$ matrices, respectively. Then $G/K$ is a symmetric space of Helgason's type ${\rm AI},$ with associated involution $\sigma$ defined by $\sigma(X) = \overline{X},$ for all $X\in G.$ Let $E_{jk}$ denote the square matrix of ${\mathfrak g}{\mathfrak l}(3)$ with entry $1$ where the $i$th row and the $j$th column meet, all other entries being $0,$ and set $A_{jk} = i(E_{jj}-E_{kk}),$ $B_{jk} = E_{jk}-E_{kj}$ and $C_{jk} = i(E_{jk} + E_{kj}).$ Then,
 \[
 {\mathfrak m} = {\mathbb R}\{A_{12},A_{23}\}\oplus {\mathbb R}\{C_{12},C_{13},C_{23}\},\quad {\mathfrak k} = {\mathbb R}\{B_{12},B_{13},B_{23}\}.
 \]
 It is clear that the space ${\mathfrak a} = \{{\rm diag}(it_{1},it_{2},it_{3}),\;t_{j}\in {\mathbb R},\;\sum_{j=1}^{3}t_{j} = 0\} = {\mathbb R}\{A_{12},A_{23}\}$ is a Cartan subspace of the space ${\mathfrak m}\subset{\mathfrak g}.$ Since ${\mathfrak a}$ is indeed a Cartan subalgebra of ${\mathfrak g},$ the root system $\Delta$ of ${\mathfrak g}$ coincides with the restricted root system $\Sigma$ and the centralizer ${\mathfrak h}$ is the zero algebra. Moreover, $H = \exp({\mathfrak a})\cap K = \{{\rm diag}(\varepsilon_{1},\varepsilon_{2},\varepsilon_{3})\colon {\varepsilon}_{j} = \pm 1,\;\Pi_{j=1}^{3}\varepsilon_{j} = 1\}.$ A system of (restricted) roots of ${\mathfrak g}^{\mathbb C}$ with respect to ${\mathfrak a}^{\mathbb C}$ is $\Sigma = \{\lambda_{1},\lambda_{2}, \lambda_{3} = \lambda_{1} + \lambda_{2}\},$ where
 \begin{equation}\label{lambdas}
 \lambda_{1}(A_{12}) = 2i,\quad \lambda_{1}(A_{23}) = -i,\quad \lambda_{2}(A_{12}) = -i,\quad \lambda_{2}(A_{23}) = 2i.
 \end{equation}
 Then, $\lambda_{3}(A_{12}) = \lambda_{3}(A_{23}) = i.$ Because the Killing form of ${\mathfrak g}$ is proportional to the trace form, we put $\langle X,Y\rangle = -\frac{1}{2}{\rm trace}XY,$ for all $X,Y\in {\mathfrak g}.$ Then $\{C_{12},C_{13},C_{23}\}$ and $\{B_{12},B_{13},B_{23}\}$ are orthonormal bases of $\sum_{j=1}^{3}{\mathfrak m}_{\lambda_{j}}\subset {\mathfrak m}$ and ${\mathfrak k},$ respectively. Moreover, $\langle A_{12},A_{12}\rangle = \langle A_{23},A_{23}\rangle = 1$ and  $\langle A_{12},A_{23}\rangle = -\frac{1}{2}.$ Hence, $\{X_{1},X_{2}\}$ is an orthonormal basis of ${\mathfrak a},$ where
 \[
 X_{1}=  \frac{\sqrt{3}}{3}(A_{12}-A_{23}),\quad X_{2} = A_{12} + A_{23},
 \]
and the restriction of $\lambda_{\mathbb R}$ to ${\mathcal S}_{W}(r),$ for each $\lambda\in \Sigma^{+},$ as function on $]0,\theta_{\rm max}[,$ is given by
\[
(\lambda_{1})_{\mathbb R}(\theta) = -r(\sqrt{3}\cos\theta + \sin\theta),\quad (\lambda_{2})_{\mathbb R}(\theta) = r(\sqrt{3}\cos\theta -\sin\theta),\quad (\lambda_{3})_{\mathbb R}(\theta) = -2r \sin\theta.
\]
Using (\ref{lambdas}), we have $W= \{ aX_{1}+bX_{2} : 0<b<\sqrt{3}a\}.$ Hence, it follows that $\theta_{\rm max} = \pi/3$ and the rays $R_{1}$ and $R_{2}$ are generated by $X_{1}$ and $X_{1} + \sqrt{3}X_{2},$ respectively. The multiplicity of $\lambda_{1},$ $\lambda_{2}$ and $\lambda_{3}$ is $1$ and we get
 $$
 \begin{array}{lcllcl}
 {\mathfrak m}_{\lambda_{1}} = {\mathbb R}\{C_{12}\},& & {\mathfrak m}_{\lambda_{2}} = {\mathbb R}\{C_{23}\},& &{\mathfrak m}_{\lambda_{3}} = {\mathbb R}\{C_{13}\},\\[0.4pc]
 {\mathfrak k}_{\lambda_{1}} = {\mathbb R}\{B_{12}\}, & & {\mathfrak k}_{\lambda_{2}} = {\mathbb R}\{B_{23}\}, & & {\mathfrak k}_{\lambda_{3}} = {\mathbb R}\{B_{13}\}.
 \end{array}
 $$

 \subsection{The Case ${\rm rank}\;G/K = {\bf r}> 2$} For each $w= \sum_{j=1}^{\bf{r}}w_{j}X_{j}\in {\mathcal S}_{W}(r),$ $r>0,$ there exists $j_{0}\in\{1,\dots ,{\mathbf r}\}$ such that $w_{j_{0}}\neq 0.$ Let $\alpha\in \{0,1\}$ such that $(-1)^{\alpha}w_{j_{0}}>0$ and consider on the open neighborhood 
 \begin{equation}\label{Ualpha}
 {\mathcal U}^{\alpha}_{j_{0}} = \{u\in {\mathcal S}_{W}(r)\colon (-1)^{\alpha}\langle u,X_{j_{0}}\rangle >0\}
 \end{equation}
 of $w$ in ${\mathcal S}_{W}(r),$ the basis of vector fields 
  \[
  \{P_{j} = x_{j}\frac{\partial}{\partial x_{j_{0}}} - x_{j_{0}}\frac{\partial}{\partial x_{j}}\;\colon  j = 1,\dots ,{\mathbf r},\, j\neq j_{0}\}.
  \]
Put $Y_{j} = x_{j}X_{j_{0}} - x_{j_{0}}X_{j}.$ Then,
\begin{equation}\label{basisxi}
\{\xi^{S}_{(o_{H},w)}, (Y_{j}(w),0),(0,P_{j}(w)), (\xi^{s}_{\lambda},0),(\zeta^{s}_{\lambda},0)\},
\end{equation}
for all $j = 1,\dots, {\mathbf r},$ $j\neq j_{0},$ and $s = 1,\dots ,m_{\lambda},$ $\lambda\in \Sigma^{+},$ is a basis of $\overline{\mathfrak m}\times T_{w}{\mathcal S}_{W}(r)\cong T_{(o_{H},w)}(G/H\times {\mathcal S}_{W}(r)).$ 

\begin{remark}{\rm Taking $j_{0} = 2$ and $j = 1,$ and considering the identification $]0,\theta_{\rm max}[ \cong {\mathcal S}_{W}(r),$ the basis (\ref{basisxi}) can be also chosen when ${\rm rank}\;G/K = 2,$ where $P_{1}$ and $Y_{1}$ are the vector fields on ${\mathcal S}_{W}(r)$ given by $P_{1}(w(\theta)) = \frac{d}{d\theta}$ and $Y_{1}(w(\theta)) = r(\cos\theta X_{2}-\sin \theta X_{1}).$
}
\end{remark}
\begin{lemma}\label{brackets} If ${\rm rank}\;G/K = {\mathbf r}\geq 2,$ then, for each $w\in {\mathcal S}_{W}(r)$ such that $w_{j_{0}}\neq 0,$ $j_{0}\in \{1,\dots,{\bf r}\},$ we have
$$
\begin{array}{l}
[\xi^{S},(Y_{j},0)]_{(o_{H},w)}  =  0, \quad [\xi^{S},(0,P_{j})]_{(o_{H},w)}  = (-\frac{1}{r}Y_{j}(w),0),\quad \mbox{$j\in\{1,\dots ,r\},\;j\neq j_{0},$} \\[0.4pc]
[\xi^{S},(\xi^{s}_{\lambda},0)]_{(o_{H},w)}  =  (-\frac{\lambda_{\mathbb R}(w)}{r}\zeta^{s}_{\lambda},0), \quad  [\xi,(\zeta^{s}_{\lambda},0)]_{(o_{H},w)} =  (\frac{\lambda_{\mathbb R}(w)}{r}\xi^{s}_{\lambda},0).
\end{array}
$$
\end{lemma}
\begin{proof} Because $(X_{j},0)(x_{k}) = 0,$ it follows that $[\xi^{S},(Y_{j},0)]_{(o_{H},w)} = 0.$ On the other hand,  we obtain
$$
\begin{array}{lcl}
[\xi^{S},(0,P_{j})]_{(o_{H},w)} & = & \frac{1}{r}\sum_{k=1}^{\bf r}[x_{k}(X_{k},0),(0,x_{j}\frac{\partial}{\partial x_{j_{0}}}-x_{j_{0}}\frac{\partial}{\partial x_{j}})]_{(o_{H},w)}\\[0.4pc]
& = & \frac{1}{r}\sum_{k=1}^{\bf r}(-w_{j}\frac{\partial x_{k}}{\partial x_{j_{0}}}(w) + w_{j_{0}}\frac{\partial x_{k}}{\partial x_{j}}(w))(X_{k},0)\\[0.4pc]
& = & \frac{1}{r}(-w_{j}(X_{j_{0}},0) + w_{j_{0}}(X_{j},0)) = -\frac{1}{r}(Y_{j}(w),0).
\end{array}
$$
Finally, using (\ref{eq.ms4.3}) and proceeding in the same way, the last two equalities are obtained. 
\end{proof}
Now we have all the necessary material to prove the main result of this section.

\begin{theorem}\label{tmain} A symmetric space $G/K$ of compact type has rank one if and only if there exists a $G$-invariant Riemannian metric $\tilde{\bf g}$ on the tangent sphere bundle $T_{r}(G/K),$ for some $r>0,$ for which the standard vector field $\xi^{S}$ is Killing. Each of these metrics $\tilde{\bf g}$ on $T_{r}(G/K) = G/H,$ where ${\rm rank}\;G/K = 1,$ is determined by
$$
\begin{array}{lcl}
\tilde{\bf g}_{o_{H}}= \left\{\begin{array}{lcl}
a_{0}^{2}\langle\cdot,\cdot\rangle_{\mathfrak a} + a_{\varepsilon}\langle\cdot,\cdot\rangle_{{\mathfrak m}_{\varepsilon}\oplus {\mathfrak k}_{\varepsilon}} &\mbox{if} & G/K = {\mathbb S}^{n}\;\mbox{or}\;\; {\mathbb R}{\mathbf P}^{n},\\[0.4pc]

a_{0}^{2}\langle\cdot,\cdot\rangle_{\mathfrak a} + a_{\varepsilon}\langle\cdot,\cdot\rangle_{{\mathfrak m}_{\varepsilon}\oplus {\mathfrak k}_{\varepsilon}} + a_{\varepsilon/2} \langle\cdot,\cdot\rangle_{{\mathfrak m}_{\varepsilon/2}\oplus {\mathfrak k}_{\varepsilon/2}}  & \mbox{if} & G/K =  {\mathbb C}{\mathbf P}^{n},\; {\mathbb H}{\mathbf P}^{n}\;\mbox{or}\; \;{\mathbb C}a{\mathbf P}^{n},
 \end{array}
 \right.
 \end{array}
 $$
for some positive constants $a_{0},a_{\varepsilon}$ and $a_{\varepsilon/2}.$
\end{theorem}
\begin{proof} If ${\rm rank}\;G/K\geq 2,$ we prove that there is no $G$-invariant Riemannian metric on $G/H\times {\mathcal S}_{W}(r),$ for each $r>0,$ for which $\xi^{S}$ is Killing. Let $\tilde{\mathbf g}$ be any $G$-invariant Riemannian metric on $G/H\times{\mathcal S}_{W}(r).$ Since the vector fields $(Y_{j},0)$ and $(0,P_{j})$ on each open subset ${\mathcal U}^{\alpha}_{j_{0}}$ of ${\mathcal S}_{W}(r),$ defined in (\ref{Ualpha}), are $G$-invariant, it follows that $\tilde{\mathbf g}((Y_{j},0),(0,P_{k})) (aH,w) = f_{jk}(w),$ for all $w\in {\mathcal U}^{\alpha}_{j_{0}},$ where $f_{jk}$ is a smooth function defined on ${\mathcal U}^{\alpha}_{j_{0}}.$ Then, using  Lemma \ref{brackets}, we get
 \begin{equation}\label{L2}
 ({\mathcal L}_{\xi^{S}}\tilde{\mathbf g})_{(o_{H},w)}((Y_{j},0),(0,P_{k} ))  =   \frac{1}{r}\tilde{\bf g}_{(o_{H},w)}((Y_{j},0),(Y_{k},0)).
\end{equation}
Hence $\xi^{S}$ cannot be Killing. If ${\rm rank}\;G/K = 1,$ from (\ref{Bb}) and using (\ref{eq.ms4.3}), we have
$$
\begin{array}{lcl}
({\mathcal L}_{\xi^{S}}\tilde{\mathbf g})_{o_{H}}(\xi^{s}_{\lambda},\zeta^{s}_{\lambda}) & =  & -\tilde{\mathbf g}_{o_{H}}([\xi^{S},(\xi^{s}_{\lambda})^{\tau_{H}}],\zeta^{s}_{\lambda}) - \tilde{\mathbf g}_{o_{H}}(\xi^{s}_{\lambda},[\xi^{S},(\zeta^{s}_{\lambda})^{\tau_{H}}])\\[0.4pc]
& = & -\tilde{\mathbf g}_{o_{H}}([X,\xi^{s}_{\lambda}]_{\overline{\mathfrak m}},\zeta^{s}_{\lambda}) - \tilde{\mathbf g}_{o_{H}}(\xi^{s}_{\lambda},[X,\zeta^{s}_{\lambda}]_{\overline{\mathfrak m}})\\[0.4pc]
&  =  & -b_{\lambda}\langle[X,\xi^{s}_{\lambda}]_{\overline{\mathfrak m}},\zeta^{s}_{\lambda}\rangle - a_{\lambda}\langle \xi^{s}_{\lambda},[X,\zeta^{s}_{\lambda}]_{\overline{\mathfrak m}}\rangle = \lambda_{\mathbb R}(X)(b_{\lambda}-a_{\lambda}),
\end{array}
$$
for all $s = 1,\dots, m_{\lambda},$ $\lambda\in\Sigma^{+}.$ Because the other components of $({\mathcal L}_{\xi^{S}}\tilde{\mathbf g})_{o_{H}}$ are zero, $\xi^{S}$ is Killing if and only if $a_{\lambda} = b_{\lambda},$ for all $\lambda\in \Sigma^{+}.$ This proves the result.
\end{proof}

As a consequence from Theorem \ref{tmain}, using (\ref{inducedgs}), we have the following.

\begin{corollary}\label{lrank1} If $G/K$ is a symmetric space of compact type, the standard vector field $\xi^{S}$ on $(T_{r}(G/K),\tilde{g}^{S})$ is Killing if and only if $r = 1$ and $G/K = {\mathbb S}^{n}$ or ${\mathbb R}{\mathbf P}^{n}.$
\end{corollary}

\section{A family of almost Hermitian structures}

Using that the subspaces ${\mathfrak m}_{\lambda}$ and ${\mathfrak k}_{\lambda},$ for each $\lambda\in \Sigma^{+},$ are ${\rm Ad}(H)$-invariant, we construct a family of $G$-invariant almost Hermitian structures $(J^{q},{\bf g} = {\bf g}^{q,a_{0},a_{\lambda}})$ on $G/H\times W,$ depending of smooth functions $q\colon {\mathbb R}^{+}\to {\mathbb R}^{+}$ and $a_{0}, a_{\lambda} \colon W\to {\mathbb R}^{+},$ for each $\lambda\in \Sigma^{+},$ as follows: 

The $G$-invariant almost complex structure $J^{q}$ at $(o_{H},w)$ coincides with $J^{S}_{(o_{H},w)}$ on ${\mathfrak a}\times T_{w}W,$ for each $w\in W,$ that is
\begin{equation}\label{Jq1}
J^{q}_{(o_{H},w)}(u,v_{w}) = (-v,u_{w}),\quad \mbox{\rm for all $u,v\in {\mathfrak a},$}
\end{equation}
and, on $\sum_{\lambda\in \Sigma^{+}}({\mathfrak m}_{\lambda}\oplus{\mathfrak k}_{\lambda})\times \{0\},$ it is expressed as
\begin{equation}\label{Jq2}
\begin{array}{l}
J^{q}_{(o_{H},w)}(\xi^{s}_{\lambda},0) =  (-\frac{1}{q_{\lambda}(w)}\zeta_{\lambda}^{s},0), \quad  J_{(o_{H},w)}(\zeta^{s}_{\lambda},0)  =  (q_{\lambda}(w)\xi^{s}_{\lambda},0),
\end{array}
\end{equation}
for all $s =1,\dots ,m_{\lambda},$ $\lambda\in \Sigma^{+},$ where $q_{\lambda} = q\circ \lambda_{\mathbb R}.$ The $G$-invariant Hermitian metric ${\bf g}= {\bf g}^{q,a_{0},a_{\lambda}}$ is determined by 
\begin{equation}\label{gnew}
{\bf g}_{(o_{H},w)} = a_{0}^{2}(w)(\pi_{1}^{*}\langle\cdot,\cdot\rangle_{\mathfrak a} + \pi_{2}^{*}\langle\cdot,\cdot\rangle_{W}) + \sum_{\lambda\in \Sigma^{+}}a_{\lambda}(w)\pi_{1}^{*}(\langle\cdot,\cdot\rangle_{{\mathfrak m}_{\lambda}} + (q_{\lambda}(w))^{2}\langle\cdot,\cdot\rangle_{{\mathfrak k}_{\lambda}})
\end{equation}
where $\pi_{1}\colon G/H\times W\to G/H$ and $\pi_{2}\colon G/H\times W\to W$ are the projections of $G/H\times W$ and $\langle\cdot,\cdot\rangle_{W}$ is the natural Riemannian metric on $W:$ $\langle(v_{1})_{w},(v_{2})_{w}\rangle_{W} = \langle v_{1},v_{2}\rangle,$ for all $v_{1},v_{2}\in {\mathfrak a}.$

 From Proposition \ref{estandard}, the standard almost Hermitian structure $(J^{S},g^{S})$ on $G/H\times W$ is the structure $(J^{q},{\bf g}^{q,a_{0},a_{\lambda}}),$ where $q = {\rm Id}_{\mathbb R}$ and $a_{0} = a_{\lambda} = 1,$ for all $\lambda\in \Sigma^{+}.$ Moreover, for $q(t) = \tanh t,$ the $G$-invariant almost complex structure $J^{q}$ coincides with the {\em canonical complex structure} $J^{c}$ on $G/H\times W^{+}$ \cite{GGM1}, i.e., at the points $(o_{H},w)\in G/H\times W^{+},$ $J^{c}$ is determined by 
 \[
\begin{array}{l}
J^{c}_{(o_{H},w)}(X_j,0)
= \Big(0,\frac{\partial}{\partial x_j}(w)\Big),\quad J^{c}_{(o_{H},w)}\Big(0,\frac{\partial}{\partial x^{j}}(w)\Big) = -(X_{j},0),
\; j=1,\dotsc ,{\bf r},\\[0.5pc]
J^{c}_{(o_{H},w)} (\xi_\lambda^s,\,  0)
=(-\coth\lambda_{\mathbb R}(w)\zeta_\lambda^s, \, 0), \quad J^{c}_{(o_{H},w)}(\zeta^{s}_{\lambda},\,0) = (\tanh\lambda_{\mathbb R}(w)\xi^{s}_{\lambda},\,0),
\end{array}
\]
where $s=1,\dotsc ,m_\lambda$ and $\lambda\in\Sigma^+.$ 

\begin{remark}{\rm Every compact, normal Riemannian homogeneous manifold $G/K,$ in particular any symmetric space of compact type, admits a (unique) {\em adapted} complex structure on its entire tangent bundle $T(G/K) = G\times_{K}{\mathfrak m}$  called the {\em canonical} complex structure \cite{Szoke1}. Under the $G$-equivariant diffeomorphism $G\times_{K}{\mathfrak m}\to G^{\mathbb C}/K^{\mathbb C},$ $[(a,x)]\mapsto a\exp(ix)K^{\mathbb C}$ (see \cite{Mo1}), where $G^{\mathbb C}$ and $K^{\mathbb C}$ are the complexifications of the corresponding Lie groups, Sz\"oke in \cite{Szoke1} obtains that $J ^{c}$ coincides precisely with the natural complex structure of $G^{\mathbb C}/K^{\mathbb C}.$}
\end{remark}
 
 By the $G$-equivariant diffeomorphism $\phi\circ \chi,$ the structure $(J^{q},{\bf g}^{q,a_{0},a_{\lambda}})$ is defined on the open dense subset $D(G/K)$ of $T(G/K).$ Next we look for necessary and sufficient conditions so that it can be extended to the entire tangent bundle $T(G/K).$

\begin{theorem}\label{TJc} The almost Hermitian structure $(J^{q}, {\bf g}^{q,a_{0},a_{\lambda}})$ on $G\times_{K}{\mathfrak m}^{R} = \chi(G/H\times W)$ can be extended to a $G$-invariant structure on the entire $T(G/K)\cong G\times_{K}{\mathfrak m}$ if and only if $\lim_{t\to 0^{+}}\frac{q(t)}{t}\in {\mathbb R}^{+}$ and, for each limit point  $\overline{w} = \lim_{m\to \infty}w_{m} \in \overline{W}\setminus W\subset {\mathfrak a}$ of some sequence $w_{m}\in W,$ $m\in {\mathbb N},$ there exist with (finite) positive values the limits $\lim_{m\to \infty}a_{0}(w_{m})$ and $\lim_{m\to \infty}a_{\lambda}(w_{m}),$ for all $\lambda\in \Sigma^{+}.$
 \end{theorem}
For its proof we will need the following lemma.

\begin{lemma}\label{l1Ad} For each $x= {\rm Ad}_{k}w\in {\mathfrak m}^{R},$ $w\in W, k\in K,$ and $(\mu,u)\in {\mathfrak m}\times{\mathfrak m},$ we have
\[
\pi_{*(a,x)}(\mu^{\tt l}_{a},u_{x}) = \pi_{*(ak,w)}(({\rm Ad}_{k^{-1}}\mu)^{\tt l}_{ak},({\rm Ad}_{k^{-1}}u)_{w}).
\]
\end{lemma}
\begin{proof} We get 
$$
\begin{array}{l} \pi_{*(a,x)}(\mu^{\tt l}_{a},u_{x})  =    \frac{d}{dt}_{\mid t = 0}\pi(a\exp t\mu, {\rm Ad}_{k}(w + t{\rm Ad}_{k^{-1}}u)) = \frac{d}{dt}_{\mid t = 0} \pi(a(\exp t\mu) k, w + t{\rm Ad}_{k^{-1}}u)\\[0.4pc]
= \frac{d}{dt}_{\mid t = 0}\pi(ak(k^{-1}(\exp t\mu) k),w + t{\rm Ad}_{k^{-1}}u) 
= \pi_{*(ak,w)}(({\rm Ad}_{k^{-1}}\mu)^{\tt l}_{ak},({\rm Ad}_{k^{-1}}u)_{w})).\
\end{array}
$$
 \end{proof}

\noindent{\bf Proof of Theorem \ref{TJc}.} From (\ref{lGm}) and applying (\ref{f++}), together with (\ref{Jq1}) and (\ref{Jq2}), the $G$-invariant almost complex structure $J^{q}= \chi_{*}J^{q}\chi^{-1}_{*}$ on $G\times_{K}{\mathfrak m}^{R}$ at the point $[(a,w)]\in \pi(G\times W),$ is determined by
$$
\begin{array}{lcl}
J^{q}\pi_{*(a,w)}((X_{j})^{\tt l}_{a},0) & = & \chi_{*(aH,w)}J^{q}((X_{j})^{\tt l}_{a},0) = \chi_{*(aH,w)}(0,\frac{\partial}{\partial x_{j}}(w))= \pi_{*(a,w)}(0, (X_{j})_{w}),\\[0.4pc]
J^{q}\pi_{*(a,w)}((\xi^{s}_{\lambda})^{\tt l}_{a},0)  &  =  & \chi_{*(aH,w)}J^{q}((\xi^{s}_{\lambda})^{\tt l}_{a},0) =  \pi_{*(a,w)}(0,\frac{\lambda_{\mathbb R}(w)}{q_{\lambda}(w)}(\xi^{s}_{\lambda})_{w}),
\end{array}
$$
for all $j = 1,\dots,r,$ $\lambda\in \Sigma^{+}$ and $s = 1,\dots, m_{\lambda}.$ Then,
\[
J^{q}\pi_{*(a,w)}(0,(X_{j})_{x}) =  \pi_{*(a,w)}(-(X_{j})^{\tt l}_{a},0),\quad
J^{q}\pi_{*(a,w)}(0,(\xi^{s}_{\lambda})_{w})  = \pi_{*(a,w)}(-\frac{q_{\lambda}(w)}{\lambda_{\mathbb R}(w)}(\xi^{s}_{\lambda})^{\tt l}_{a},0).
\]
Hence, for all $(\mu,u)\in {\mathfrak m}\times{\mathfrak m},$
$$
\begin{array}{lcl}
J^{q}\pi_{*(a,w)}(\mu^{\tt l}_{a},u_{w}) & = & \pi_{*(a,w)} \Big(-\sum_{j=1}^{\bf r}\langle u,X_{j}\rangle(X_{j})^{\tt l}_{a} - \sum_{\lambda\in \Sigma^{+}} \frac{q_{\lambda}(w)}{\lambda_{\mathbb R}(w)} \sum_{s=1}^{m_{\lambda}}\langle u,\xi^{s}_{\lambda}\rangle(\xi^{s}_{\lambda})^{\tt l}_{a},\\[0.4pc]
& & \sum_{j=1}^{\bf r}\langle \xi,X_{j}\rangle(X_{j})_{w}+ \sum_{\lambda\in \Sigma^{+}}\frac{\lambda_{\mathbb R}(w)}{q_{\lambda}(w)}\sum_{s=1}^{m_{\lambda}}\langle\xi,\xi^{s}_{\lambda}\rangle(\xi^{s}_{\lambda})_{w} \Big).
\end{array}
$$
Now, using Lemma \ref{l1Ad} and that $\langle\cdot,\cdot\rangle$ is ${\rm Ad}(K)$-invariant, $J^{q}$ at the point $[(a,x)],$ where $x = {\rm Ad}_{k}w,$ $w\in W$ and $k\in K,$ is determined by
\begin{equation}\label{ff1}
\begin{array}{lcl}
J^{q}\pi_{*(a,x)}(\mu^{\tt l}_{a},u_{x})& = & \pi_{*(ak,w)} \Big(-\sum_{j=1}^{\bf r}\langle u,{\rm Ad}_{k}X_{j}\rangle(X_{j})^{\tt l}_{ak}\\[0.4pc]
& &  \hspace{1.5 cm}- \sum_{\lambda\in \Sigma^{+}} \frac{q_{\lambda}(w)}{\lambda_{\mathbb R}(w)}\sum_{s=1}^{m_{\lambda}}\langle u,{\rm Ad}_{k}\xi^{s}_{\lambda}\rangle(\xi^{s}_{\lambda})^{\tt l}_{ak},\\[0.4pc]
& & \hspace{1.5cm}  \sum_{j=1}^{r}\langle \mu,{\rm Ad}_{k}X_{j}\rangle(X_{j})_{w}\\[0.4pc]
& &  \hspace{1.5cm} + \sum_{\lambda\in \Sigma^{+}}\frac{\lambda_{\mathbb R}(w)}{q_{\lambda}(w)}\sum_{s=1}^{m_{\lambda}}\langle\mu,{\rm Ad}_{k}\xi^{s}_{\lambda}\rangle(\xi^{s}_{\lambda})_{w})\Big).
\end{array}
\end{equation}

On the other hand, using (\ref{f++}) and (\ref{gnew}), the Riemannian metric ${\bf g}$ on $G\times_{K}{\mathfrak m}^{R},$ under the relationship ${\bf g} = \chi^{*}{\bf g},$ is determined at the point $[(a,w)]\in \pi(G\times W)$ by
$$
\begin{array}{l}
{\bf g}(\pi_{*}((X_{j})^{\tt l}_{a},0),\pi_{*}((X_{j})^{\tt l}_{a},0)) = {\bf g}(\pi_{*}(0,(X_{j})_{w}),\pi_{*}(0,(X_{j})_{w})) = a_{0}^{2}(w),\\[0.4pc]
{\bf g}(\pi_{*}((\xi^{s}_{\lambda})^{\tt l}_{a},0),\pi_{*}((\xi^{s}_{\lambda})^{\tt l}_{a},0)) = \frac{(\lambda_{\mathbb R}(w))^{2}}{(q_{\lambda}(w))^{2}}{\bf g}(\pi_{*}(0,(\xi^{s}_{\lambda})_{w}),\pi_{*}(0,(\xi^{s}_{\lambda})_{w}) = a_{\lambda}(w),
\end{array}
$$
where $j = 1,\dots, {\bf r}$ and $s = 1,\dots, m_{\lambda},$ $\lambda\in \Sigma^{+},$ the rest of components being zero. Hence, applying again Lemma \ref{l1Ad}, we get
\begin{equation}\label{gg}\begin{array}{l}
{\bf g}(\pi_{*}(\mu^{\tt l}_{a},u_{x}),\pi_{*}(\nu^{\tt l}_{a},v_{x}))  = \\[0.4pc]
\hspace{3cm} a_{0}^{2}(w)\sum_{j=1}^{\bf r}(\langle\mu,{\rm Ad}_{k}X_{j}\rangle\langle\nu,{\rm Ad}_{k}X_{j}\rangle + \langle u,{\rm Ad}_{k}X_{j}\rangle\langle v,{\rm Ad}_{k}X_{j}\rangle)\\[0.4pc]
\hspace{3cm} + \sum_{\lambda\in \Sigma^{+}}a_{\lambda}(w)\Big(\sum_{s = 1}^{m_{\lambda}}(\langle\mu,{\rm Ad}_{k}\xi^{s}_{\lambda}\rangle\langle\nu,{\rm Ad}_{k}\xi^{s}_{\lambda}\rangle \\[0.4pc]
\hspace{3cm}+ \frac{(q_{\lambda}(w))^{2}}{(\lambda_{\mathbb R}(w))^{2}}\sum_{s = 1}^{m_{\lambda}}(\langle u,{\rm Ad}_{k}\xi^{s}_{\lambda}\rangle\langle v,{\rm Ad}_{k}\xi^{s}_{\lambda}\rangle)\Big).
\end{array}
\end{equation}

Because ${\rm Ad}(K)W = {\mathfrak m}^{R}$ and ${\rm Ad}(K)(\overline{W}) = {\mathfrak m},$ consider for each point $\overline{w}\in \overline{W}\setminus W\subset {\mathfrak a},$ some sequence $w_{m}\in W,$ $m\in {\mathbb N},$ such that $\lim_{m \to \infty}w_{m} = \overline{w}.$ From (\ref{ff1}) and (\ref{gg}), the existence of the extension of $(J^{q},{\bf g})$ is determined by the existence with (finite) positive value of the limit $ \lim_{m\to \infty}\frac{q(\lambda_{\mathbb R}(w_{m}))}{\lambda_{\mathbb R}(w_{m})},$ $\lim_{m\to \infty}a_{0}(w_{m})$ and $\lim_{m\to \infty}a_{\lambda}(w_{m}),$ for each $\lambda\in \Sigma^{+}.$ Since $\lim_{m\to \infty}\lambda_{\mathbb R}(w_{m}) = \lambda_{\mathbb R}(\overline{w}) = 0,$ the result is proved.

\begin{remark}{\rm Many functions $q\colon {\mathbb R}^{+}\to {\mathbb R}^{+},$ such as $\sinh Ct,$ $\ln(1+ Ct),$ $e^{Ct}-1,$ $\tanh Ct$ and $Ct,$ for a constant $C>0,$ as linear combinations $C_{1}f_{1} + C_{2}f_{2}$ of them, for all $C_{1},C_{2}>0,$ satisfy $\lim_{t\to0^{+}}\frac{q(t)}{t} \in {\mathbb R}^{+}.$ Indeed, if $\overline{q}\colon ]-\delta,\infty[\to {\mathbb R},$ for some $\delta>0,$ is a smooth extension of a smooth function $q\colon {\mathbb R}^{+}\to {\mathbb R}^{+},$ then $\lim_{t\to 0^{+}}\frac{q(t)}{t}\in {\mathbb R}^{+}$ if and only if $\overline{q}(0) = 0$ and $\frac{d\overline{q}}{dt}(0)\in {\mathbb R}^{+}.$}
 \end{remark}

\begin{theorem}\label{tunique} In the set of all (extended) almost complex structures $\{J^{q}\}
$ of $T(G/K),$ with smooth function $q$ satisfying $\lim_{t\to 0^{+}}\frac{q(t)}{t}\in {\mathbb R}^{+},$ the canonical complex structure $J^{c} = J^{q},$ with $q(t) = \tanh t,$ is the unique structure which is integrable.
\end{theorem}
 \begin{proof} Suppose that $J^{q}$ is a structure of the family which is integrable, that is, its Nijenhuis torsion $[J^{q},J^{q}]$ vanishes. Because $J^{q}((X_{j})^{\tau_{H}},0) = (0, \frac{\partial}{\partial x_{j}})$ and $J^{q}((\xi^{s}_{\lambda})^{\tau_{H}},0) = (-\frac{1}{q_{\lambda}} (\zeta^{s}_{\lambda})^{\tau_{H}},0)$ on ${\mathcal U}_{o_{H}}\times W,$  we have, using (\ref{eq.ms4.3}), that 
$$
\begin{array}{lcl}
[((X_{j})^{\tau_{H}},0), ((\xi^{s}_{\lambda})^{\tau_{H}},0)]_{(o_{H},w)} &  = & (-\lambda_{\mathbb R}(X_{j})\zeta^{s}_{\lambda},0),\\[0.4pc]
J^{q}[J^{q}(X_{j})^{\tau_{H}},0),((\xi^{s}_{\lambda})^{\tau_{H}},0)]_{(o_{H},w)} &  = & (0,0_{w}),\\[0.4pc]
J^{q}[((X_{j})^{\tau_{H}},0),J^{q}((\xi^{s}_{\lambda})^{\tau_{H}},0)]_{(o_{H},w)} & = & (\frac{\lambda_{\mathbb R}(X_{j})}{(q_{\lambda}(w))^{2}}\zeta^{s}_{\lambda},0).\\[0.4pc]
\end{array}
$$
Moreover, since $\lambda_{\mathbb R}$ is a linear mapping, one gets that $\frac{\partial\lambda_{\mathbb R}}{\partial x_{j}} = \lambda_{\mathbb R}(X_{j})$ and then
$$
\begin{array}{lcl}
 [J^{q}((X_{j})^{\tau_{H}},0),J^{q}((\xi^{k}_{\lambda})^{\tau_{H}},0)]_{(o_{H},x)}  & = &( -\frac{\partial}{\partial x_{j}}(\frac{1}{q_{\lambda}})(w)\zeta^{k}_{\lambda},0)=(\frac{\lambda_{\mathbb R}(X_{j})\frac{dq}{dt}(\lambda_{\mathbb R}(w))}{(q_{\lambda}(w))^{2}}\zeta^{k}_{\lambda},0).
 \end{array}
$$
Hence, 
\[
[J^{q},J^{q}]_{(o_{H},w)}((X_{j},0),(\xi^{k}_{\lambda},0))= \frac{\lambda_{\mathbb R}(X_{j})}{(q_{\lambda}(w))^{2}}\Big( 1-  (q(\lambda_{\mathbb R}(w)))^{2} - \frac{dq}{dt}(\lambda_{\mathbb R}(w))\Big)
\]
and $[J^{q},J^{q}]_{(o_{H},w)}((X_{j},0),(\xi^{s}_{\lambda},0))=0,$ for all $w\in W,$ $j = 1,\dots ,r$ and $s = 1,\dots ,m_{\lambda},$ $\lambda\in \Sigma^{+},$ if and only if the Riccati's equation 
\[
\frac{dq}{dt} = 1-q^{2}
\]
holds. Putting $z = \frac{1}{q-1},$ we obtain the linear equation $z' = 2z +1$ whose general solution is given by $z(t) = Ce^{-2t}-\frac{1}{2}.$ Therefore, we have
\[
q(t) = \frac{C e^{2t} + \frac{1}{2}}{C e^{2t}-\frac{1}{2}},
\]
for a constant $C$ which satisfies $|C|\geq \frac{1}{2},$ using that $q(t)> 0,$ for all $t\in {\mathbb R}^{+}.$ In particular, for $C = -\frac{1}{2},$ one gets $q(t) = \tanh t$ and $\lim_{t\to 0^{+}}\frac{q(t)}{t} = 1.$ For $C = \frac{1}{2},$ one gets $q(t) = \coth t$ and so, $\lim_{t\to 0^{+}}\frac{q(t)}{t} = +\infty.$ For $C\neq \frac{1}{2}$ and $C\neq -\frac{1}{2},$ $q(t)$ can be expressed as $q(t) = \frac{C+1/2}{C-1/2} - \frac{2C}{(C-1/2)^{2}}t + O(t^{2}),$ for sufficiently small $t>0.$ Hence, $\lim_{t\to 0^{+}}\frac{q(t)}{t} = +\infty.$ Therefore, $q(t) = \tanh t.$ .
\end{proof}
 
 Finally, we determine those almost Hermitian structures $(J^{q},{\bf g}^{q,a_{0},a_{\lambda}})$ on $T(G/K)$ whose fundamental $2$-form is, up to a scalar, the symplectic form $d\theta.$ To do this, we first prove the following lemma.
  
 \begin{lemma}\label{ldtheta} The pull-back form $\theta:=\chi^{*}\theta$ of the canonical $1$-form $\theta$ on $G/H\times W$ and its exterior differential $d\theta$ are determined by
$$
\begin{array}{l}
\theta_{(o_{H},w)}(\mu,u_{w}) = \langle w,\mu\rangle,\\[0.4pc]
d\theta_{(o_{H},w)}((\mu,u_{w}),(\nu,v_{w}))  =  \frac{1}{2}(\langle u,\nu\rangle-\langle v,\mu\rangle - \langle [w, \mu],\nu\rangle),
\end{array}
$$
for all $(\mu,u_{w}),$ $(\nu,v_{w})\in \overline{\mathfrak m}\times T_{w}W,$ $w\in W.$
\end{lemma}
\begin{proof} From (\ref{1canonical}) and (\ref{f+}),
\[
\theta_{(o_{H},w)}(\mu,u_{w}) = \theta_{[(e,w)]}\pi_{*(e,w)}(\mu_{\mathfrak m},u_{w} + [\mu_{\mathfrak k},w]_{w}) = \langle w, \mu_{\mathfrak m}\rangle = \langle w,\mu\rangle.
\]
 This proves the first equality. For the second, consider, for each $w = \sum_{j=1}^{\bf r}w_{j}X_{j}\in W^{+}$ and $u = \sum_{j=1}^{\bf r}u_{j}X_{j}\in {\mathfrak a},$ the vector $u_{w}\in T_{w}W^{+}$ given by $u_{w} = \sum_{j=1}^{\bf r}u_{j}\frac{\partial}{\partial x_{j}}(w)$ and also denote by $u$ the vector field on $W^{+}$ with constant components, $u = \sum_{j=1}^{\bf r}u_{j}\frac{\partial}{\partial x_{j}}.$ Then, $\theta(\mu^{\tau_{H}}, u) = \sum_{j=1}^{\bf r}\langle \mu,X_{j}\rangle x_{j}$ and $(\mu,u_{w})\theta(\nu^{\tau_{H}},v) = \langle u,\nu\rangle.$ Hence, the equality is obtained using the Maurer-Cartan formula.
 \end{proof}

\begin{theorem}\label{tKahler} For each smooth function $q\colon {\mathbb R}^{+}\to {\mathbb R}^{+}$ such that $\lim_{t\to 0^{+}}\frac{q(t)}{t}\in {\mathbb R}^{+},$ the almost Hermitian structure $(J^{q},{\bf g}^{q,a_{0},a_{\lambda}})$ on $G/H\times W$ can be extended to an almost K\"ahler structure on all of $T(G/K),$ where $a_{0}$ is a constant function and, for each $\lambda\in \Sigma^{+},$
\begin{equation}\label{alambda}
a_{\lambda}(w) = \frac{a_{0}^{2}\lambda_{\mathbb R}(w)}{q_{\lambda}(w)},\quad w\in W.
\end{equation}
 Moreover, it is K\"ahler if and only if $q(t) = \tanh t,$ i.e., $J^{q}$ is the canonical complex structure. 
\end{theorem}
\begin{proof} Using (\ref{eq.ms4.3}) in Lemma \ref{ldtheta}, $d\theta$ on $G/H\times W$ is determined by 
\begin{equation}\label{componentstheta}
\begin{array}{lcl}
d\theta_{(o_{H},w)}((X_{j},0),(0,\frac{\partial}{\partial x_{j}})) & = & -\frac{1}{2},\quad j = 1,\dots,{\bf r},\\[0,4pc]
d\theta_{(o_{H},w)}((\xi^{s}_{\lambda},0),(\zeta^{s}_{\lambda},0)) & = & \frac{1}{2}\lambda_{\mathbb R}(w),\quad s = 1,\dots ,m_{\lambda},\:\lambda\in \Sigma^{+},
\end{array}
\end{equation}
the rest of components of $d\theta_{(o_{H},w)}$ being zero. On the other hand, the fundamental $2$-form $\omega :={\bf g}(\cdot,J^{q}\cdot)$ at $(o_{H},w)$ for ${\bf g} = {\bf g}^{q,a_{0},a_{\lambda}},$ has the following non-zero components:
 $$
 \begin{array}{lcl}
 \omega_{(o_{H},w)}((X_{j},0),(0,\frac{\partial}{\partial x_{j}}(w)) = -a_{0}^{2}(w),\quad j = 1,\dots,{\bf r},\\[0.4pc]
 \omega_{(o_{H},w)}((\xi^{s}_{\lambda},0),(\zeta^{s}_{\lambda},0))= a_{\lambda}(w)q_{\lambda}(w),\quad s = 1,\dots,m_{\lambda},\;\lambda\in \Sigma^{+}.
 \end{array}
 $$
 Hence, from (\ref{componentstheta}) and using (\ref{alambda}), we have that $\omega = 2a_{0}^{2}d\theta.$ Since $a_{0}$ is a constant function, it is a symplectic structure and, using Theorem \ref{TJc}, $(J^{q},{\bf g}^{q,a_{0},a_{\lambda}})$ is extended to an almost K\"ahler structure throught $T(G/K).$ For the last part, we use directly Theorem \ref{tunique}.
\end{proof}
 
\begin{remark}\label{ralmostkahler}{\rm In the above theorem, for each smooth function $q,$ the almost K\"ahler metrics obtained are, using (\ref{gnew}), homothetic with coefficient $a_{0}^{2}$ to the metric ${\bf g}^{q}$ on $G/H\times W,$ extendible to $T(G/K)$ and determined by \begin{equation}\label{gq}
{\bf g}^{q} = {\bf g}^{q,1,\frac{\lambda_{\mathbb R}}{q_{\lambda}}}.
\end{equation}
In particular, if $q = {\rm Id}_{{\mathbb R}^{+}} ,$ ${\bf g}^{q}$ is the Sasaki metric $g^{S}.$}
\end{remark}

\section{Contact metric structures on tangent sphere bundles}

Before starting with the contents of this section, let's briefly recall some basic concepts about almost contact metric structures. For more information, see \cite{Bl}. An odd-dimensional smooth manifold $M$ is called {\em almost contact} if it admits a $(\varphi,\xi,\eta)$-structure, where $\varphi$ is a tensor field of type $(1,1),$ $\xi$ is a vector field and $\eta$ is a $1$-form, such that
\[
\varphi^{2} = -{\rm I} + \eta\otimes \xi,\quad \eta(\xi) = 1.
\]
Then $\varphi\xi = 0$ and $\eta\circ\varphi =0.$ If $M$ is equipped with a Riemannian metric $g$ such that
\[
g(\varphi X,\varphi Y) = g(X,Y) - \eta(X)\eta(Y),
\]
for all vector fields $X,Y$ on $M,$ $(M,\varphi,\xi,\eta,g)$ is said to be an {\em almost contact metric manifold} and $g$ is called a {\em compatible metric}. Moreover, we say that $(\varphi,\xi,\eta,g)$ is {\em contact metric} if $d\eta(X,Y) =g(X,\varphi Y).$ Then a contact metric structure is determined by the pair $(\xi,g)$ (or $(\eta,g))$ and $\xi$ is known as its {\em characteristic vector field}. If, in addition, $\xi$ is a Killing vector field, then the manifold is called {\em $K$-contact}. Recall that an almost contact structure $(\varphi,\xi,\eta)$ is said to be {\em normal} if the $(1,1)$-tensor field $[\varphi,\varphi](X,Y) + 2{\rm d}\eta(X,Y)\xi$ vanishes for all vector fields $X,Y.$ A contact metric structure $(\xi, g)$ which is normal is called a {\em Sasakian structure}. Any Sasakian manifold is always $K$-contact. The converse holds for the three-dimensional case, but it may be not true in higher dimension. 

 Here we look for $G$-invariant Riemannian metrics $\tilde{\bf g}$ on $D_{r}(G/K)\cong G/H\times {\mathcal S}_{W}(r)$ (see Proposition \ref{ll1}), such that the pair $(\xi = \frac{1}{\kappa}\xi^{S},\tilde{\bf g})$  is a contact metric structure, for some smooth function $\kappa\colon {\mathcal S}_{W}(r)\to {\mathbb R}^{+}.$ Each of these structures will be expressed as induced from $(J^{q},{\bf g}^{q,a_{0},a_{\lambda}}),$ for certain smooth functions $q,a_{0}$ and $a_{\lambda},$ $\lambda\in \Sigma^{+}.$ 
 
  Let us first suppose that $G/K$ is a compact rank-one symmetric space.
 
  \begin{theorem}\label{mainrank1} If ${\rm rank}\;G/K = 1,$ any $G$-invariant Riemannian metric $\tilde{\bf g}$ on the tangent sphere bundle $T_{r}(G/K)= G/H$ of radius $r>0,$ such that the pair $(\xi = \frac{1}{\kappa}\xi^{S},\tilde{\bf g})$ is a contact metric structure for some constant $\kappa>0,$ is determined at the origin $o_{H}\in G/H$ by
  \begin{equation}\label{gc}
 \tilde{\bf g}_{o_{H}} = \kappa^{2}\langle\cdot,\cdot\rangle_{\mathfrak a} + \frac{\kappa}{2q_{\varepsilon}}\langle\cdot,\cdot\rangle_{{\mathfrak m}_{\varepsilon}} + \frac{\kappa}{4q_{\varepsilon/2}}\langle\cdot,\cdot\rangle_{{\mathfrak m}_{\varepsilon/2}} + \frac{\kappa q_{\varepsilon}}{2}\langle\cdot,\cdot\rangle_{{\mathfrak k}_{\varepsilon}}  + \frac{\kappa q_{\varepsilon/2}}{4}\langle\cdot,\cdot\rangle_{{\mathfrak k }_{\varepsilon/2}},
 \end{equation} 
 for some positive constants $q_{\lambda},$ $\lambda\in \Sigma^{+}.$ Moreover, $(\xi,\tilde{\bf g})$ is $K$-contact, indeed Sasakian, if and only if $q_{\lambda} = 1,$ for each $\lambda\in \Sigma^{+},$ or equivalently 
 \begin{equation}\label{eqsasakian}
 \tilde{\bf g}_{o_{H}} = \kappa^{2}\langle\cdot,\cdot\rangle_{\mathfrak a} + \frac{\kappa}{2}\langle\cdot,\cdot\rangle_{{\mathfrak m}_{\varepsilon}\oplus{\mathfrak k}_{\varepsilon}}+\frac{\kappa}{4}\langle\cdot,\cdot\rangle_{{\mathfrak m}_{\varepsilon/2}\oplus{\mathfrak k}_{\varepsilon/2}}.
 \end{equation}
 \end{theorem}
\begin{proof} Each $G$-invariant Riemannian metric $\tilde{\bf g}$ on $G/H$ is expressed at the origin $o_{H}$ as in (\ref{Bb}), for positive constants $a_{0},$ $a_{\lambda}$ and $b_{\lambda},$ for $\lambda\in \Sigma^{+}.$ Because $\tilde{\bf g}(\xi,\xi) = 1,$ it follows that $a_{0}=\kappa$ and the one-form defined by $\eta = \tilde{\bf g}(\xi,\cdot)$ is the $G$-invariant one-form on $G/H$ such that $\eta_{o_{H}} = \kappa \eta^{S}_{o_{H}}.$ Then, for $u,v\in \overline{\mathfrak m},$ we have
$$
\begin{array}{lcl}
d\eta_{o_{H}}(u,v) & = & d\eta_{o_{H}}(u^{\tau_{H}},v^{\tau_{H}}) = -\frac{1}{2}\eta_{o_{H}}([u^{\tau_{H}},v^{\tau_{H}}]) \\[0.4pc]
& =&  -\frac{1}{2}\eta_{o_{H}}([u,v]_{\overline{\mathfrak m}}) = -\frac{\kappa}{2}\langle [X,u],v\rangle,
\end{array}
$$
where $u^{\tau_{H}}$ and $v^{\tau_{H}}$ denote the local vector fields on $G/H$ around $o_{H}$ with $u^{\tau_{H}}_{o_{H}} = u$ and $v^{\tau_{H}}_{o_{H}} = v,$ defined in (\ref{tautau}). Hence, using (\ref{eq.ms4.3}), it follows that
\[
d\eta_{o_{H}}(\xi^{s}_{\lambda},\zeta^{s}_{\lambda}) = \frac{\kappa\lambda_{\mathbb R}(X)}{2}, \quad s = 1,\dots ,m_{\lambda},\; \lambda\in \Sigma^{+},
\]
the rest of components of $d\eta_{o_{H}}$ being zero. Then, from (\ref{Bb}), the $G$-invariant $(1,1)$-tensor field $\varphi$ such that $d\eta = \tilde{\bf g}(\cdot,\varphi \cdot)$ is determined at $o_{H}$ by
\[
\varphi_{o_{H}}\xi^{s}_{\lambda} = -\frac{\kappa \lambda_{\mathbb R}(X)}{2b_{\lambda}}\zeta^{s}_{\lambda},\quad \varphi_{o_{H}}\zeta^{s}_{\lambda} = \frac{\kappa\lambda_{\mathbb R}(X)}{2a_{\lambda}}\xi^{s}_{\lambda}.
\]
Since $\varphi^{2} = -I + \eta\otimes \xi,$ we have that
\begin{equation}\label{ab}
b_{\lambda} = \frac{\kappa^{2}(\lambda_{\mathbb R}(X))^{2}}{4a_{\lambda}}
\end{equation}
and then,
\begin{equation}\label{varphivarphi}
 \begin{array}{lclclcl}
 \varphi_{o_{H}}\xi^{s}_{\lambda} & = & -\frac{1}{q_{\lambda}}\zeta^{s}_{\lambda}, & & \varphi_{o_{H}} \zeta^{s}_{\lambda} & = & q_{\lambda}\xi^{s}_{\lambda},
 \end{array}
 \end{equation}
 for $q_{\lambda} = \frac{\kappa\lambda_{\mathbb R}(X)}{2a_{\lambda}}.$ Therefore, using (\ref{ab}) on (\ref{Bb}), we obtain (\ref{gc}) and the almost contact metric structure $(\varphi,\xi,\eta,\tilde{\bf g})$ is then contact metric.

 The second part of the theorem follows from Theorem \ref{tmain}, taking into account that $\xi = \frac{1}{\kappa}\xi^{S}$ is Killing if and only if $a_{\lambda} = b_{\lambda},$ which is equivalent, using (\ref{ab}), to be $q_{\lambda}$ equals to $1,$ for each $\lambda\in \Sigma^{+}.$ This proves (\ref{eqsasakian}). Finally, in \cite[Theorem 1.1]{JC}, the author has proved that $(\xi = \frac{1}{\kappa}\xi^{S},\tilde{\bf g}),$ where $\tilde{\bf g}$ satisfies (\ref{eqsasakian}) and $\kappa>0,$ is indeed Sasakian.
\end{proof}
 
For a general contact metric structure $(\xi,g),$ the tensor field $h$ given by $h = \frac{1}{2}{\mathcal L}_{\xi}\varphi$ enjoys many important properties. In particular, $h$ vanishes if and only if $\xi$ is Killing, $h\xi = 0,$ $h$ anti-commutes with $\varphi$ and ${\rm tr}\;h = 0.$ Moreover, for the Levi-Civita connection $\nabla$ of $g,$
\begin{equation}\label{nablanabla}
 \nabla_{X}\xi = -\varphi X - \varphi hX.
 \end{equation}
 
 Using (\ref{eq.ms4.3}) and (\ref{varphivarphi}) and also (\ref{nablanabla}), we obtain the following.
\begin{proposition} Under the conditions of {\rm Theorem \ref{mainrank1}}, the tensor field $h$ associated to the $G$-invariant contact metric structure $(\xi = \frac{1}{\kappa}\xi^{S},\tilde{\bf g})$ is determined by
\[
h(\xi^{s}_{\lambda}) = \frac{\lambda_{\mathbb R}(X)}{2\kappa q_{\lambda}}(q^{2}_{\lambda}-1)\xi^{s}_{\lambda},\quad h(\zeta^{s}_{\lambda})  = \frac{\lambda_{\mathbb R}(X)}{2\kappa q_{\lambda}}(1- q^{2}_{\lambda})\zeta^{s}_{\lambda},
\]
for all $\lambda \in \Sigma^{+}$ and $s = 1,\dots,m_{\lambda}.$ Moreover, the Levi-Civita connection $\nabla^{\tilde{\bf g}}$ of $\tilde{\bf g},$ satisfies
\[
\nabla^{\tilde{\bf g}}_{\xi^{s}_{\lambda}}\xi = \frac{1}{q_{\lambda}}\Big(1 + \frac{\lambda_{\mathbb R}(X)}{2\kappa q_{\lambda}}(q^{2}_{\lambda}-1)\Big)\zeta^{s}_{\lambda},\quad \nabla^{\tilde{g}}_{\zeta^{s}_{\lambda}}\xi = -q_{\lambda}\Big(1 + \frac{\lambda_{\mathbb R}(X)}{2\kappa q_{\lambda}}(1- q^{2}_{\lambda})\Big)\xi^{s}_{\lambda}.
\]
\end{proposition}

Since, for each $r>0$ and any symmetric space $G/K$ of compact type, the outward unit normal vector field on the hypersurface $\iota_{r}\colon G/H\times {\mathcal S}_{W}(r)\to (G/H\times W, {\bf g}^{q,a_{0},a_{\lambda}}),$ at a point $(o_{H},w),$ $w\in {\mathcal S}_{W}(r),$ is $\frac{1}{a_{0}(w)}N_{(o_{H},w)},$ the induced almost contact metric structure $(\varphi^{q},\xi,\eta, \tilde{\bf g}^{q,a_{0},a_{\lambda}})$ on $G/H\times {\mathcal S}_{W}(r)$ from $(J^{q},{\bf g}^{q,a_{0},a_{\lambda}})$ is determined by 
$$
\begin{array}{l}
\xi_{(o_{H},w)} = \frac{1}{a_{0}(w)}\xi^{S}_{(o_{H},w)} = \frac{1}{a_{0}(w)r}(w,0);\\[0.4pc]
 \eta_{(o_{H},w)}(u,v_{w}) = a_{0}(w)\eta^{S}(u,v_{w}) = \frac{a_{0}(w)}{r}\langle u,w\rangle,\quad (u,v_{w})\in \overline{\mathfrak m}\times T_{w}({\mathcal S}_{W}(r));\\[0.4pc]
{\varphi}^{q}_{(o_{H},w)}(u,v_{w}) = (-v,u_{w} - \frac{1}{r^{2}}\langle u,w\rangle w_{w}),\quad \mbox{\rm for all}\; (u,v_{w})\in {\mathfrak a}\times T_{w}({\mathcal S}_{W}(r)),\\[0.4pc]
\varphi^{q}_{(o_{H},w)}(\xi^{s}_{\lambda},0) = (-\frac{1}{q_{\lambda}(w)}\zeta^{s}_{\lambda},0),\quad \varphi^{q}_{(o_{H},w)}(\zeta^{s}_{\lambda},0) = (q_{\lambda}(w)\xi^{s}_{\lambda},0),\quad s =1,\dots ,m_{\lambda},$ $\lambda\in \Sigma^{+};\\[0.4pc]
\tilde{\bf g}^{q,a_{0},a_{\lambda}}_{(o_{H},w)}  = a_{0}^{2}(w)(\pi_{1}^{*}\langle\cdot,\cdot\rangle_{\mathfrak a} + \pi_{2}^{*}\langle\cdot,\cdot\rangle_{{\mathcal S}_{W}(r)}) + \sum_{\lambda\in \Sigma^{+}}a_{\lambda}(w)\pi_{1}^{*}(\langle\cdot,\cdot\rangle_{{\mathfrak m}_{\lambda}} + (q_{\lambda}(w))^{2}\langle\cdot,\cdot\rangle_{{\mathfrak k}_{\lambda}}) .
\end{array}
$$

 \begin{remark}\label{rrank1}{\rm In Theorem \ref{mainrank1}, each contact metric structure $(\xi = \frac{1}{\kappa}\xi^{S},\tilde{\bf g})$ on $G/H,$ can be seen as induced from $(J^{q}, {\bf g}^{q,a_{0},a_{\lambda}}),$ where $q$ is any smooth function $q\colon {\mathbb R}^{+}\to {\mathbb R}^{+}$ such that $q(\lambda_{\mathbb R}(r)) = q_{\lambda}$ and $a_{0}$ and $a_{\lambda}$ are the constant functions $a_{0}(t) = \kappa$ and $a_{\lambda}(t) = \frac{\kappa\lambda_{\mathbb R}(X)}{2q_{\lambda}},$ for each $\lambda\in \Sigma^{+}.$ }
 \end{remark}

 The subset $\overline{{\mathcal S}_{W}(r)}\setminus {\mathcal S}_{W}(r)$ of limit points of ${\mathcal S}_{W}(r)$ is nonempty when ${\rm rank}\;G/K\geq 2.$ The following result is a direct consequence from the proof of Theorem \ref{TJc}.  
 
 \begin{lemma}\label{cTJc} If ${\rm rank}\,G/K\geq 2,$ the induced almost contact metric structure $(\varphi^{q},\xi,\eta,\tilde{\bf g}^{q,a_{0},a_{\lambda}})$ on $G\times_{K}{\mathcal S}_{\mathfrak m}^{R}(r) = \chi(G/H\times {\mathcal S}_{W}(r)),$ for some $r>0,$ can be extended to whole $T_{r}(G/K)$ if and only if $\lim_{t\to 0^{+}}\frac{q(t)}{t}\in {\mathbb R}^{+}$ and, for each limit point  $\overline{w} = \lim_{m\to \infty}w_{m} \in \overline{{\mathcal S}_{W}(r)}\setminus {\mathcal S}_{W}(r)$ of some sequence $w_{m}\in {\mathcal S}_{W}(r),$ $m\in {\mathbb N},$ the limits $lim_{m\to \infty}a_{0}(w_{m})$ and $\lim_{m\to \infty}a_{\lambda}(w_{m}),$ for each $\lambda\in \Sigma^{+},$ take (finite) positive values.
 \end{lemma} 
 
 \begin{remark}{\rm If ${\rm rank}\;G/K = 2,$ the restrictions of the functions $a_{0},$ $\lambda_{\mathbb R},$ $q_{\lambda}$ and $a_{\lambda},$ for each $\lambda\in \Sigma^{+},$ to ${\mathcal S}_{W}(r),$ can be seen as functions on the open interval $]0,\theta_{\rm max}[.$ Hence, $(\varphi^{q},\xi,\eta,\tilde{\bf g}^{q,a_{0},a_{\lambda}})$ can be extended to $T_{r}(G/K)$ if and only if $\lim_{t\to 0^{+}}\frac{q(t)}{t},$ $\lim_{\theta\to 0^{+}}a_{0}(\theta),$ $\lim_{\theta\to \theta_{\rm max}^{-}}a_{0}(\theta),$ $\lim_{\theta\to 0^{+}}a_{\lambda}(\theta)$ and $\lim_{\theta\to\theta^{-}_{\rm max}}a_{\lambda}(\theta)$ take (finite) positive values.}
 \end{remark}

 \begin{theorem}\label{tcontact} The pair $(\xi= \frac{1}{a_{0}}\xi^{S},\tilde{\bf g}^{q,a_{0},a_{\lambda}})$ is a contact metric structure on $G/H\times {\mathcal S}_{W}(r),$ for each $r>0,$ if and only if 
 \begin{enumerate}
 \item[{\rm (i)}] ${a_{0}}_{\mid {\mathcal S}_{W}(r)}$ is a constant function, equal to $\frac{1}{2r}$ for ${\rm rank}\;G/K \geq 2,$ and 
 \item[{\rm (ii)}] the following equality
\[
a_{\lambda}(w) = \frac{a_{0}(w)\lambda_{\mathbb R}(w)}{2rq_{\lambda}(w)}
\]
holds for all $w\in {\mathcal S}_{W}(r).$
\end{enumerate}
Moreover, it is $K$-contact, indeed Sasakian, if and only if ${\rm rank}\;G/K = 1$ and $q_{\lambda}(r) = 1,$ for each $\lambda\in \Sigma^{+}.$ For ${\rm rank}\;G/K\geq 2,$ $(\xi,\tilde{\bf g}^{q,a_{0},a_{\lambda}})$ can be extended to the entire $T_{r}(G/K)$ if and only if $\lim_{t\to 0^{+}}\frac{q(t)}{t}\in {\mathbb R}^{+}.$
\end{theorem}
\begin{proof} If ${\rm rank}\;G/K = 1,$ the result follows directly from Theorem \ref{mainrank1} and Remark \ref{rrank1}, together with Theorem \ref{tmain}.
 
 Next, suppose that ${\rm rank}\;G/K= {\bf r} \geq 2$ and let $w\in {\mathcal U}^{\alpha}_{j_{0}}\subset {\mathcal S}_{W}(r),$ for some $\alpha\in\{0,1\}$ and $j_{0}\in \{1,\dots,{\bf r}\},$ where $w_{j_{0}} = \langle w,X_{j_{0}}\rangle \neq 0$ and ${\mathcal U}^{\alpha}_{j_{0}}$ is the open subset defined in (\ref{Ualpha}). Consider the basis of $\overline{\mathfrak m}\times T_{w}{\mathcal S}_{W}(r)\cong T_{(o_{H},w)}(G/H\times {\mathcal S}_{W}(r))$ given in (\ref{basisxi}). Then, putting $\tilde{\bf g} =  \tilde{\bf g}^{q,a_{0},a_{\lambda}},$ we get
\begin{equation}\label{gY}
\tilde{\bf g}_{(o_{H},w)}((Y_{j},0),(Y_{k},0)) = \tilde{\bf g}_{(o_{H},w)}((0,P_{j}),(0,P_{k})) = a_{0}^{2}(w)(w_{j_{0}}^{2}\delta_{jk} + w_{j}w_{k}),
\end{equation}
for all $j,k\in \{1,\dots,{\bf r}\},$ $j, k\neq j_{0}.$ Moreover, 
\begin{equation}\label{qY}
\begin{array}{l}
\eta_{(o_{H},w)}(Y_{j}(w),0) = \eta_{(o_{H},w)}(0,P_{j}(w)) = 0,\\[0.4pc]
\varphi^{q}_{(o_{H},w)}(Y_{j}(w),0) = (0,P_{j}(w)),\quad \varphi^{q}_{(o_{H},w)}(0,P_{j}(w)) = (-Y_{j}(w),0).
\end{array}
\end{equation}
Now, we calculate $d\eta_{(o_{H},w)}(\xi,(0,P_{j}),$ $d\eta_{(o_{H},w)}((Y_{j},0),(0,P_{k})),$ for $j,k=1,\dots,{\bf r},$ $j,k\neq j_{0},$ and $d\eta_{(o_{H},w)}((\xi^{s}_{\lambda},0),(\zeta^{s}_{\lambda},0)),$ for $s = 1,\dots m_{\lambda},$ $\lambda\in \Sigma^{+}.$ The rest of components of $d\eta_{(o_{H},w)},$ with respect to this basis, are zero.
$$
\begin{array}{lcl}
d\eta_{(o_{H},w)}(\xi,(0,P_{j})) & =  &-\frac{1}{2}\eta_{(o_{H},w)}([\xi,(0,P{j})])\\[0.4pc]
& =  & -\frac{1}{2r}\sum_{k=1}^{\bf r}\eta_{(o_{H},w)} ([(\frac{x_{k}}{a_{0}}(X_{k})^{\tau_{H}},0),(0,x_{j}\frac{\partial}{\partial x_{j_{0}}}-x_{j_{0}}\frac{\partial}{\partial x_{j}})])\\[0.5pc]
\hspace{1.5cm}&  =  & -\frac{1}{2r}\sum_{k=1}^{\bf r} (-w_{j}\frac{\partial(\frac{x_{k}}{a_{0}})}{\partial x_{j_{0}}}(w) + w_{j_{0}}\frac{\partial(\frac{x_{k}}{a_{0}} )}{\partial x_{j}}(w))\eta_{(o_{H},w)}(X_{k},0).
\end{array}
$$
Hence, because $\eta_{(o_{H},w)}(X_{k},0) = \frac{a_{0}(w)}{r}w_{k}$ and $\sum_{k=1}^{\bf r}w_{k}^{2} = r^{2},$ we get
\begin{equation}\label{q1}
 d\eta_{(o_{H},w)}(\xi,(0,P_{j}))=  \frac{1}{2a_{0}(w)}(w_{j_{0}}\frac{\partial a_{0}}{\partial x_{j}}(w) - w_{j}\frac{\partial a_{0}}{\partial x_{j_{0}}}(w)).
\end{equation}
$$
\begin{array}{lcl}
d\eta_{(o_{H},w)}((Y_{j},0),(0,P_{k})) & = & -\frac{1}{2}\eta_{(o_{H},w)}([((Y_{j})^{\tau_{H}},0),(0,P_{k})])\\[0.4pc]
& =  & -\frac{1}{2}\eta_{(o_{H},w)}(w_{j_{0}}\delta_{jk}X_{j_{0}}+ w_{k}X_{j},0) = -\frac{a_{0}(w)}{2r}(w_{j_{0}}^{2}\delta_{jk} + w_{j}w_{k}).
\end{array}
$$
Then, from (\ref{gY}),
\begin{equation}\label{q2}
d\eta_{(o_{H},w)}((Y_{j},0),(0,P_{k})) = -\frac{1}{2r a_{0}(w)}\tilde{\bf g}_{(o_{H},w)}((Y_{j},0),(Y_{k},0)).
\end{equation}
Finally, $d\eta_{(o_{H},w)}((\xi^{s}_{\lambda},0),(\zeta^{s}_{\lambda},0))  =  -\frac{1}{2}\eta_{(o_{H},w)}(([\xi^{s}_{\lambda},\zeta^{s}_{\lambda}]_{\overline{\mathfrak m}},0) = -\frac{a_{0}(w)}{2 r}\langle\xi^{s}_{\lambda},[\zeta^{s}_{\lambda},w]\rangle.$ Now, applying (\ref{eq.ms4.3}), we get
\begin{equation}\label{q3}
d\eta_{(o_{H},w)}((\xi^{s}_{\lambda},0),(\zeta^{s}_{\lambda},0))= \frac{a_{0}(w)}{2r}\lambda_{\mathbb R}(w).
\end{equation}
On the other hand, using (\ref{qY}), it follows that
\[
\tilde{\bf g}((Y_{j},0),\varphi^{q}(0,P_{k}))  =  - \tilde{\bf g}((Y_{j},0),(Y_{k},0)),\quad 
\tilde{\bf g}((\xi^{s}_{\lambda},0),\varphi^{q}(\zeta^{s}_{\lambda},0)) =  q_{\lambda}(w)a_{\lambda}(w),
\]
the rest of components of the $2$-form $\tilde{\bf g}(\cdot,\varphi^{q}\cdot)$ being zero. Therefore, applying (\ref{q1}), (\ref{q2}) and (\ref{q3}), $(\varphi^{q},\xi,\eta,\tilde{\bf g})$ is contact metric if and only if (i) and (ii) hold. The last part of the Theorem follows as a direct consequence of Theorem \ref{tmain} and Theorem \ref{mainrank1} and applying Lemma \ref{cTJc}.  \end{proof}

For each almost K\"ahler metric ${\bf g}^{q}$ in (\ref{gq}), Theorem \ref{tcontact} leads to the following result.

\begin{corollary}\label{cmain} The induced $G$-invariant structure on $T_{r}(G/K),$ $r>0,$ obtained from the almost K\"ahler structure $(J^{q},\frac{1}{4r^{2}}{\bf g}^{q})$ on $T(G/K),$ is contact metric, for each smooth function $q\colon {\mathbb R}^{+}\to {\mathbb R}^{+}$ such that $\lim_{t\to 0^{+}}\frac{q(t)}{t}\in {\mathbb R}^{+}.$
\end{corollary}
Another consequence of Theorem \ref{tcontact} is the following version of Tashiro's Theorem \cite{Tash} for tangent sphere bundles of {\em any} radius of symmetric spaces of compact type.

\begin{corollary} The standard almost contact metric structure on $T_{r}(G/K)$ is contact metric if and only if $r = \frac{1}{2}.$ Moreover, for all $r>0,$ the rectified almost contact metric structure $(\varphi = \varphi^{S},\xi = 2r\xi^{S},\eta = \frac{1}{2r}\eta^{S},\tilde{g} = \frac{1}{4r^{2}}\tilde{g}^{S})$ is contact metric and it is $K$-contact if and only if $r = 1$ and $G/K = {\mathbb S}^{n}$ or ${\mathbb R}{\bf P}^{n};$ in that case $(\varphi,\xi,\eta,\tilde{g})$ is Sasakian.
\end{corollary}

Finally, we analyze the condition for $\xi$ to be Killing with respect to $\tilde{\bf g}^{q,a_{0},a_{\lambda}}$ and the normality of the almost contact structure $(\varphi^{q},\xi,\eta)$ on $G/H\times {\mathcal S}_{W}(r).$

\begin{theorem} On the almost contact metric manifold $(G/H\times {\mathcal S}_{W}(r),\varphi^{q},\xi,\eta,\tilde{\bf g}^{q,a_{0},a_{\lambda}}),$ for each $r>0,$ the following statements are equivalent:
\begin{enumerate}
\item[{\rm (i)}] the vector field $\xi$ is Killing;
\item[{\rm (ii)}] the almost contact structure $(\varphi^{q},\xi,\eta)$ is normal;
\item[{\rm (iii)}] ${\rm rank}\;G/K = 1$ and $q_{\lambda}(r) = 1,$ for each $\lambda\in \Sigma^{+}.$
\end{enumerate}
\end{theorem}
\begin{proof} ${\rm (i)}\Rightarrow {\rm (iii)}.$ Suppose that ${\rm rank}\;G/K = {\bf r} \geq 2$ and let $w\in {\mathcal U}^{\alpha}_{j_{0}}\subset {\mathcal S}_{W}(r),$ for some $\alpha\in \{0,1\}$ and $j_{0}\in \{1,\dots, {\bf r}\},$ where $w_{j_{0}} = \langle w,X_{j_{0}}\rangle \neq 0.$ Then, using (\ref{L2}) and (\ref{gY}),
$$
\begin{array}{lcl}
({\mathcal L}_{\xi}\tilde{\bf g})_{(o_{H},w)}((Y_{j}(w),0),(0,P_{k}(w))) & = & \frac{1}{a_{0}(w)}({\mathcal L}_{\xi^{S}}\tilde{\bf g})_{(o_{H},w)}((Y_{j}(w),0),(0,P_{k}(w))\\[0.4pc]
& =  & \frac{1}{ra_{0}(w)}\tilde{\bf g}_{(o_{H},w)}((Y_{j},0),(Y_{k},0)),
\end{array}
$$
for all $j,k\in \{1,\dots,{\bf r}\},$ $j,k\neq j_{0}.$ Hence, $\xi$ cannot be Killing. Therefore, if $\xi$ Killing then ${\rm rank}\;G/K = 1,$ $\xi = \frac{1}{a_{0}(r)}\xi^{S}$ on $G/H\times \{r\}$ and so the standard vector field $\xi^{S}$ is also Killing. Then, using Theorem \ref{tmain}, we have that $q_{\lambda}(r) = 1,$ for each $\lambda\in \Sigma^{+}.$

${\rm (iii)}\Rightarrow {\rm (i)}.$ We have that $a_{\lambda}(r) = b_{\lambda}(r)$ in (\ref{Bb}) and from Theorem \ref{tmain}, $\xi^{S}$ and $\xi$ are Killing vector fields.

${\rm (ii)}\Rightarrow {\rm (iii)}.$ If ${\rm rank}\;G/K = {\bf r}\geq 2,$ it follows, using Lemma \ref{brackets} and (\ref{qY}) that, 
\[
[\varphi^{q},\varphi^{q}]_{(o_{H},w)}(\xi,(Y_{j},0)) + 2d\eta_{(o_{H},w)}(\xi,(Y_{},0)) = \frac{1}{ra_{0}(w)}(0,P_{j}(w)),\quad j = 1,\dots ,{\bf r},\;j \neq j_{0},
\]
for all $w\in {\mathcal S}_{W^{+}}(r)$ such that $x_{j_{0}}(w) \neq 0,$ for some $j_{0}\in \{1,\dots ,{\bf r}\}.$ Hence, $(\varphi^{q},\xi,\eta)$ cannot be normal. If ${\rm rank}\;G/K = 1,$ applying (\ref{eq.ms4.3}), we obtain
 \[
 [\varphi^{q},\varphi^{q}]_{(o_{H} }(\xi,\xi^{s}_{\lambda}) + 2d\eta_{o_{H}}(\xi,\xi^{s}_{\lambda})\xi = \frac{\lambda_{\mathbb R}(X)}{a_{0}q_{\lambda}^{2}(r)}(q_{\lambda}^{2}(r) -1)\zeta^{s}_{\lambda}, \quad s = 1,\dots, m_{\lambda},\quad \lambda\in \Sigma^{+}.
 \]
 This means that the normality of $(\varphi^{q},\xi,\eta)$ implies that $q_{\lambda},$ where $q_{\lambda} = q(\lambda_{\mathbb R}(r)),$ must be equal to $1$ for each $\lambda\in \Sigma^{+}.$ For the converse, ${\rm (iii)}\Rightarrow {\rm (ii)},$ we use the fact that $(\varphi^{1},\xi,\eta)$ is normal as the author proved in \cite{JC}. 
  \end{proof}

\end{document}